\documentclass[a4paper,11pt,reqno]{amsart}

\usepackage{latexsym}
\usepackage{amssymb,amsmath}
\usepackage{graphics}
\usepackage{url}
\usepackage[vcentermath,enableskew]{youngtab}
  
\usepackage{booktabs}  

\newtheorem{theorem}{Theorem}[section]
\newtheorem*{theorem*}{Theorem}
\newtheorem{corollary}[theorem]{Corollary}

\newtheorem{proposition}[theorem]{Proposition}
\newtheorem{conjecture}[theorem]{Conjecture}

\newtheorem{lemma}[theorem]{Lemma}
\newtheorem{problem}[theorem]{Problem}

\usepackage{caption}
\theoremstyle{definition}
\newtheorem{definition}[theorem]{Definition}

\newtheorem{example}[theorem]{Example}

\newcommand{\bz}{\mathbf{2}} %{\underline{2}} --doesn't work in tableaux very well
\newcommand{\bd}{\mathbf{3}} %{\underline{3}}

\newcommand{\N}{\mathbb{N}}
\newcommand{\Z}{\mathbb{Z}}

\newcommand{\R}{R}

\newcommand{\CC}{S}

\renewcommand{\epsilon}{\varepsilon}

\DeclareMathOperator{\Sym}{Sym}

\DeclareMathOperator{\sgn}{sgn}

\DeclareMathOperator{\height}{ht}

\renewcommand{\theta}{\vartheta}
\newcommand{\id}{\mathrm{id}}

\newcounter{thmlistcnt}
\newenvironment{thmlist}%
	{\setcounter{thmlistcnt}{0}%
	\begin{list}{\emph{(\roman{thmlistcnt})}}{%
		\usecounter{thmlistcnt}%
		\setlength{\topsep}{0pt}%
		\setlength{\leftmargin}{0pt}%
		\setlength{\itemsep}{0pt}%
		\setlength{\labelwidth}{17pt}
		\setlength{\itemindent}{30pt}}%
	}%
	{\end{list}}%

\newcommand{\nur}{{\bigl( \nu(0), \ldots, \nu(r-1) \bigr)}}

\newcommand{\nub}{{\pmb{\nu}}}
\newcommand{\sigmab}{{\pmb{\sigma}}}
\newcommand{\taub}{{\pmb{\tau}}}

\newcommand{\nurns}{{(\nu(0),\ldots,\nu(r-1))}}
\newcommand{\Pnu}{\nub^\star}
\newcommand{\RT}{\mathrm{RT}}
\newcommand{\w}{\mathrm{w}}
\newcommand{\rRTg}[1]{\text{$#1$-}\mathrm{RT}}
\newcommand{\rRT}{\rRTg{r}}
\newcommand{\RSYT}{\mathrm{RSYT}}
\newcommand{\SSYT}{\mathrm{SSYT}}
\newcommand{\SSYTL}{\mathrm{SSYTL}}
\newcommand{\RNT}{\mathrm{RNT}}

\newcommand{\RSYTL}{\mathrm{RSYTL}}
\newcommand{\mSSYT}{\mathbf{SSYT}}
\newcommand{\mSSYTL}{\mathbf{SSYTL}}

\newcommand{\cont}{\mathrm{cont}}

\newcommand{\T}{\mathrm{T}}
\newcommand{\negvthinspace}{\hskip-0.75pt}
\newcommand{\vthinspace}{\hskip0.25pt}

\newcommand{\mbf}[1]{\boldsymbol{#1}}

\newcommand{\rhonu}{\nu}
\newcommand{\alphaepsilon}{\epsilon}
\newcommand{\sigmanu}{\nu}
\newcommand{\sigmanub}{\nub}


\begin{document}

\begin{abstract}
This paper proves a
 combinatorial rule expressing the product $s_\tau(s_{\lambda/\mu} \circ p_r)$
of a Schur function and the 
plethysm of a skew Schur function with a power sum symmetric function
as an integral linear combination of Schur functions. This generalizes
the SXP rule for the plethysm $s_\lambda \circ p_r$.
Each step in the proof uses either an explicit bijection or a sign-reversing
involution. The proof is inspired by an earlier proof of the SXP rule due to Remmel and
Shimozono, \emph{A simple proof of the Littlewood--Richardson rule and applications}, 
Discrete Mathematics \textbf{193} (1998) 257--266. 
The connections with two 
later combinatorial rules for special cases of this plethysm are discussed. Two open problems
are raised. The paper is intended to be readable by
non-experts. 
%The connections with two other combinatorial rules for this plethysm are 
%discussed in the final section and an open problem is suggested.
\end{abstract}

\author{Mark Wildon}
\title[A generalized
SXP rule]{A generalized
SXP rule proved by bijections and involutions}
\date{\today}

\maketitle
\thispagestyle{empty}

\section{Introduction}

Let $f \circ g$ denote the plethysm of the symmetric functions $f$ and $g$.
%Let $s_\lambda$ denote the Schur functions labelled by the partition $\lambda$
%and let $p_r$ denote the power sum symmetric function for $r \in \N$.
While it remains a hard problem to express an
arbitrary plethysm as an integral linear combination of Schur functions, many results
are known in special cases. In particular, the SXP rule, first proved in 
\cite[page~351]{LittlewoodModular} and later, in a different way, in
\cite[pages~135--140]{ChenGarsiaRemmel}, gives a surprisingly simple formula for the
plethysm $s_\lambda \circ p_r$ where $s_\lambda$ is the Schur function for
 the partition $\lambda$
and $p_r$ is the power sum symmetric function for $r \in \N$. 
It states that %if $\lambda$ is a partition of $n$ then
\begin{equation}
\label{eq:SXP}  s_\lambda \circ p_r = \sum_{\nub}\sgn_r (\Pnu) %_\varnothing)  
c^\lambda_{\nub}\hskip0.5pt %\nurns} 
%\bigl( P(\nub) \bigr) 
s_{\Pnu} %_\varnothing}
\end{equation}
where the sum is over all $r$-multipartitions $\nub = \nur$ of $n$, 
$\Pnu$ %_\varnothing$ 
is the partition with empty $r$-core and $r$-quotient $\nub$,  
$\sgn_r (\Pnu) %\bigl( P(\nub) \bigr) 
\in \{+1,-1\}$ is as defined in \S2 below, and $c^\lambda_\nub = c^\lambda_\nurns$
is a generalized Littlewood--Richardson coefficient, as defined at the end of \S\ref{sec:LS} below.

In this note we prove a generalization of the
SXP rule. The following definition is required: say that
the pair of $r$-multipartitions $(\nub, \taub)$, denoted $\nub/\taub$,
%denoted $\nub / \taub$,
is a \emph{skew $r$-multipartition}
of $n$ if $\nu(i)/\tau(i)$ is a skew partition for each $i \in \{0,\ldots, r-1\}$, 
and $n = \sum_{i=0}^{r-1} \bigl( |\nu(i)| - |\tau(i)| \bigr)$. 
%We denote such a pair by $\nub / \taub$.

\begin{theorem}\label{thm:main}
Let $r \in \N$,
let $\tau$ be a partition with $r$-quotient $\taub$, and % = (\tau(0),\ldots, \tau(r-1))$,
let $\lambda / \mu$ be a skew partition of $n$. Then
\[ s_\tau(s_{\lambda / \mu} \circ p_r) = \sum_{\nub}\sgn_r \bigl((\nub/\taub, \tau)^\star\bigr) 
c^\lambda_{\nub/\taub \vthinspace:\vthinspace \mu} 
 %\bigl( P(\nub) \bigr) 
s_{(\nub/\taub, \tau)^\star} \]
where the sum is over all $r$-multipartitions $\nub$
such that  $\nub / \taub$ is a skew $r$-multipartition of~$n$, $(\nub/\taub,\tau)^\star$ is the partition,
defined formally in Definition~\ref{def:star},
obtained from $\tau$ by adding $r$-hooks in the way specified by $\nub / \taub$, and
$\nub/\taub : \mu$ is the skew $(r+1)$-multipartition obtained from $\nub/\taub$ by appending $\mu$.
\end{theorem}

Each step in the proof %, given in \S\ref{sec:proof} below, 
uses either an explicit bijection or a sign-reversing involution on suitable sets of tableaux.
%The first step uses the Jacobi--Trudi formula, which has a beautiful involutive proof
%given in \cite[page 342]{StanleyII}.
The critical second step uses a special case of a rule for multiplying a Schur function by the plethysm
$h_\alpha \,\circ\, p_r$, where~$h_\alpha$ is the complete symmetric function for
 the composition~$\alpha$.
This rule was first proved in \cite[page 29]{DesarmenienLeclercThibon} and is stated here
as Proposition~\ref{prop:bs}. A reader familiar with the basic results on symmetric functions
and willing to assume this rule should find the proof largely self-contained.
 In particular,  we do not assume the Littlewood--Richardson rule.
We show in \S\ref{subsec:nonPlethystic} that
two versions of the Littlewood--Richardson rule follow from Theorem~\ref{thm:main} by setting \hbox{$r=1$}
and taking either $\tau$ or $\mu$ to be the empty partition.
The penultimate step in our proof is~\eqref{eq:afterInvol}, which restates
Theorem~\ref{thm:main} in a form free from explicit
Littlewood--Richardson coefficients. 
%and specializing $\tau$ and $\mu$ 
In \S\ref{subsec:plethystic} we discuss the connections with other combinatorial rules
for plethysms of the type in Theorem~\ref{thm:main}, including the domino tableaux
rule for $s_\tau(s_\lambda \circ p_2)$ proved in~\cite{CarreLeclerc}.

%;
%alternatively, assuming the Littlewood--Richardson rule, it is routine
%to deduce Theorem~\ref{thm:main} from~\eqref{eq:SXP}.

An earlier proof of both the Littlewood--Richardson rule and  the SXP rule,
as stated in~\eqref{eq:SXP}, was given by Remmel and
Shimozono in \cite{RemmelShimozono}, using a involution on semistandard skew tableaux
defined by Lascoux and
Sch{\"u}tzenberger in \cite{LascouxSchutzenberger}. The proof given here uses
the generalization of this involution to tuples of semistandard tableaux of skew shape. 
We include full details to make the paper self-contained, while admitting that this
generalization is implicit in \cite{LascouxSchutzenberger} and \cite{RemmelShimozono}, since,
as illustrated after Example~\ref{ex:invol}, a tuple of skew tableaux
may be identified (in a slightly artificial way) with a single skew tableau.
%, given in \S 3 below.
The significant departure from the 
proof in~\cite{RemmelShimozono} is that we 
replace monomial symmetric functions with complete
symmetric functions. This dualization requires different ideas. It appears to offer
some simplifications, as well as leading to a more general result. 
%requires some new ideas and leads to
%some  simplifications. 

The plethysm operation $\circ$ is defined
in \cite[\S 2.3]{LoehrRemmel}, or, with minor changes in notation, in
\cite[I.8]{MacDonald}, \cite[A2.6]{StanleyII}. For plethysms of the form $f \circ p_r$
the definition can be given in a simple way: write $f$ as a formal infinite
sum of monomials in the variables $x_1, x_2, \ldots $ and substitute $x_i^r$ for each $x_i$
to obtain $f \circ p_r$. For example, $s_{(2)} \circ p_2 = x_1^4 + x_2^4 + x_3^4 + \cdots + x_1^2x_2^2 + x_1^2x_3^2
+ x_2^2x_3^2 + \cdots =
s_{(4)} - s_{(3,1)} + s_{(2,2)}$. By \cite[page 167, P1]{LoehrRemmel},
$f \circ p_r = p_r \circ f$;
several of the formulae we use are stated in the literature in this equivalent form.

\subsubsection*{Outline}
The necessary background results on quotients of skew partitions and ribbon tableaux are given in \S\ref{sec:background} below, where we also
recall
the plethystic Murnaghan--Nakayama rule and the Jacobi--Trudi formula.
%; these are used in the first
%and second steps of the proof. 
In~\S\ref{sec:LS}
we give a generalization of the Lascoux--Sch{\"u}tzenberger involution and define
the generalized Littlewood--Richardson coefficients appearing in Theorem~\ref{thm:main}. The proof
of Theorem~\ref{thm:main} is then given in \S\ref{sec:proof}. 
An example is given in~\S\ref{sec:example}. 
Further examples and connections with other
combinatorial rules are given in \S\ref{sec:SXPp}. In particular we deduce 
the Littlewood--Richardson rule as stated in \cite[Definition~16.1]{James} and, originally, 
in \cite[Theorem~III]{LittlewoodRichardson}. In the appendix we prove a `shape-content' involution
that implies the version of the Littlewood--Richardson rule proved in \cite{StembridgeLR},
and  a technical result motivating Conjecture~\ref{conj:RNT}.

%In \S\ref{sec:proof} we present the proof in `one-line' form. We then
%give the bijections  and involutions that prove  each step.

\section{Prerequisites on $r$-quotients, ribbons and tableaux}
\label{sec:background}

We assume the reader is familiar with partitions, skew partitions and border strips,
as defined in \cite[Chapter 7]{StanleyII}. 
Fix $r \in \N$ throughout this section.
We represent partitions using an $r$-runner abacus, as defined in \cite[page 78]{JK},
on which the number of beads is always a multiple of $r$;
the $r$-quotient of a partition is then unambiguously defined by \cite[2.7.29]{JK}.
(See \S\ref{subsec:plethystic} for a remark on this convention.)
%For $m \in \N$, we define a \emph{skew $m$-multipartition} of $n \in \N_0$ to be a tuple
%$\bigl( \nu(0)/\tau(0), \ldots, \nu(m-1)/\tau(m-1) \bigr)$ of skew partitions
%of total size $n$. 
%When $\tau(i)$ is the empty partition, 
%we write $\nu(i)$ rather than $\nu(i)/\varnothing$.
%(Where a skew $m$-multipartition corresponds to the runners of
%an abacus, we choose to number the components from $0$.)
%For further background on these objects we refer the reader to \cite{JK}
%or \cite{WildonPlethysticMN}. 
The further unnumbered definitions below 
are taken from \cite[page 28]{DesarmenienLeclercThibon}, \cite[\S 3]{EvseevPagetWildon}
and \cite[\S 3]{RemmelShimozono}, and are included to make this note self-contained.

\subsubsection*{Signs and quotients of skew partitions}
Let $n \in \N_0$ and let $\nu /\tau$ be a skew partition of $rn$.
 We say that $\nu / \tau$ is \emph{$r$-decomposable} 
if there exist partitions 
\[  \tau = \sigma^{(0)} \subset \sigma^{(1)} \subset \ldots \subset \sigma^{(n)} = \nu \]
such that $\sigma^{(j)}/\sigma^{(j-1)}$ is a border strip of size $r$
(also called an $r$-border strip) for each $j \in \{1,\ldots,n\}$. 
In this case we define the \emph{$r$-sign} of $\nu/\tau$ by
\[ \sgn_r(\nu/\tau) = \prod_{i=1}^n (-1)^{\height(\sigma^{(j)}/\sigma^{(j-1)})}. \]
(Here $\height(\sigma^{(j)}/\sigma^{(j-1)})$ is the height of the border strip $\sigma^{(j)}/\sigma^{(j-1)}$, defined to be one less than the number of rows of $\sigma^{(j)}$ that it meets.)
By \cite[2.7.26]{JK} or \cite[Proposition~3]{WildonPlethysticMN}, 
this definition is independent of the choice of the $\sigma^{(j)}$.
If $\nu /\tau$ is not $r$-decomposable, we set $\sgn_r(\nu/\tau) = 0$.

If $\nu / \tau$ is $r$-decomposable then it is possible to obtain
an abacus for $\nu$ by starting with an abacus for $\tau$
and making $n$ single-step downward bead moves. 
It follows that if $\nur$ is the $r$-quotient
of $\nu$ and
$\bigl( \tau(0), \ldots, \tau(r-1) \bigr)$ is the $r$-quotient of $\tau$
then $\nu(i) / \tau(i)$ is a skew partition for each $i$. We define
the \emph{$r$-quotient} of $\nu / \tau$, denoted $\nub/\taub$, to be the skew $r$-multipartition
$\bigl( \nu(0) / \tau(0), \ldots, \nu(r-1) / \tau(r-1) \bigr)$. 
%(Skew $r$-multipartitions
%were defined before Theorem~\ref{thm:main}.)
Conversely, the pair $(\nub/\taub, \tau)$ determines $\nu$.

%Conversely, given a skew $r$-multipartition $\nub /\taub$ as above,
%we define $\nub^\star_\tau$ to be $\nu$.

\begin{definition}\label{def:star}
Let $\tau$ be a partition with $r$-quotient $\bigl( \tau(0), \ldots, \tau(r-1) \bigr)$ 
and let $\nub / \taub$ be a skew $r$-multipartition of $n$. We define
$(\nub/\taub, \tau)^\star$ to be the unique partition $\nu$ such that $\nu/\tau$ is a
skew partition of $rn$ with $r$-quotient~$\nub/\taub$.
\end{definition}

Working with abaci with $6$ beads, we have
$\bigl( \bigl( (1), \varnothing, (2,1) / (1) \bigr), (3,2) \bigr)^\star = (6,5,2,1)$
as shown in Figure~1 above,
$\bigl( \bigl( (1), \varnothing, (2,1) / (1) \bigr), (3) \bigr)^\star = (6,2,2,2)$
and
$\bigl( \bigl( (1), (2), (1) / (1) \bigr),$ $(3,2) \bigr)^\star = (4,4,4,1,1)$.
Here we use the convention that a skew partition $\nu / \varnothing$ is written simply as $\nu$.

%As seen here, we omit the empty partition in $(1) / \varnothing$ when writing
%skew partitions.

%$\bigl( \bigl( (2) / (1), (1), (2,1) / (1) \bigr), (3,2,1) \bigr)^\star = (6,5,4,3)$,
%as shown in Figure~1 overleaf,
%$\bigl( \bigl( (2) / (1), (1), (2,1) / (1) \bigr), (4,4) \bigr)^\star = (7,7,2,2,2)$
%and
%$\bigl( \bigl( (1) / (1), (2), (2,1) / (1) \bigr), (3,2,1) \bigr)^\star = (6,6,3,3)$.

\subsubsection*{Ribbons}
Let $\rhonu / \sigma$ be a border strip in the partition $\rhonu$. If row $a$ is the least numbered
row of $\rhonu$ meeting $\rhonu / \sigma$ then we say that $\rhonu / \sigma$ has \emph{row number} $a$
and write $\R(\rhonu / \sigma) = a$.
Let $r \in \N$ and $q \in \N_0$. A skew partition $\rhonu / \tau$ of $rq$ is a \emph{horizontal $r$-ribbon strip} %,
%or \emph{$r$-ribbon} for short, 
if there exist partitions
\begin{equation}
\label{eq:strips} 
 \tau = \sigma^{(0)} \subset \sigma^{(1)} \subset \ldots \subset \sigma^{(q)} = \rhonu 
\end{equation}
such that $\sigma^{(j)}/\sigma^{(j-1)}$ is an $r$-border strip for each $j \in \{1,\ldots,q\}$ and
\begin{equation}
\label{eq:rownumbers}  
R(\sigma^{(1)} / \sigma^{(0)}) \ge \ldots \ge R(\sigma^{(q)} / \sigma^{(q-1)}). 
\end{equation}
For examples see Figure~1 above and Figure~3 in \S\ref{sec:example}.

\begin{figure}[t]
\begin{center}
\includegraphics{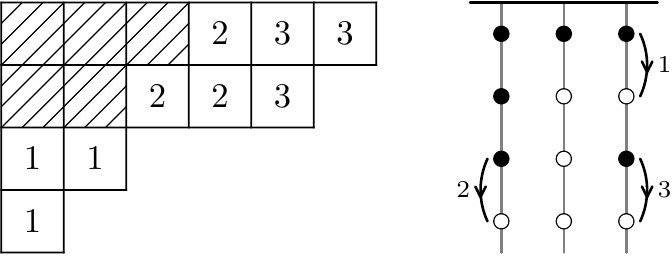}
\caption{\small The skew partition $(6,5,2,1) / (3,2)$ is a horizontal $3$-ribbon strip of size $9$, 
with $\sigma^{(1)} = (3,2,2,1)$ and $\sigma^{(2)} = (4,4,2,1)$. The
border strip $\sigma^{(i)}/\sigma^{(i-1)}$ is marked $i$; the row-numbers are~$3$,~$1$ and~$1$,
in increasing order of $i$.
The corresponding bead moves
on an abacus representing $(3,2)$ are  shown; note these satisfy the condition in Lemma~\ref{lemma:ribboneqv}(ii). The $3$-quotient of $(6,5,2,1) / (3,2)$ is $\bigl( (1), \varnothing, 
(2,1) / (1) \bigr)$, and so
$\bigl( \bigl( (1), \varnothing, (2,1) / (1) \bigr), (3,2) \bigr)^\star = (6,5,2,1)$.}
%\end{caption}
\end{center}
\end{figure}

The following lemma, which is used implicitly in \cite{DesarmenienLeclercThibon},
is needed in the proof of Theorem~\ref{thm:main}. 
%Informally,~(iii) % and should clarify
%this definition. In particular (ii) 
%says that the $\mu^{(i)}$ are uniquely determined
%by $\sigma / \mu$, and 
Informally, (iii) says that the border strips forming a horizontal $r$-ribbon strip are uniquely
determined by its shape. % the shape of the ribbon strip.
Note also that (iv) explains the sense in which horizontal $r$-ribbon strips are `horizontal'.

\begin{lemma}\label{lemma:ribboneqv}
Let $q \in \N_0$ and let $\rhonu / \tau$ be a skew partition of $rq$. The following are equivalent:
\begin{thmlist}
\item $\rhonu / \tau$ is a horizontal $r$-ribbon strip;
\item if $A$ is an abacus representing $\tau$ then,
for each $i \in \{0,1,\ldots,r-1\}$, there exists $c \in \N_0$ and unique positions 
$\beta_1, \ldots, \beta_{c}$ and 
$\gamma_1, \ldots, \gamma_{c}$  on runner~$i$ of $A$ with
\[ \beta_ 1 < \gamma_1 < \ldots < \beta_{c} < \gamma_{c} \]
such that moving the bead in position $\beta_j$ down to 
the space in position $\gamma_j$, for each~$j \in \{1,\ldots,c\}$ and $i \in \{0,1,\ldots,r-1\}$, gives
an abacus representing~$\rhonu$;
%\item if $B$ is an abacus representing $\rhonu$ then,
%for each $i \in \{0,1,\ldots,r-1\}$, there exist unique positions $\alpha_1, \ldots, \alpha_{c(i)}$ 
%and $\beta_1, \ldots, \beta_{c(i)}$ on runner $i$ with
%\[ \alpha_ 1 < \beta_1 < \ldots < \alpha_{c(i)} < \beta_{c(i)} \]
%such that moving the bead in position $\beta_j$ to the space in position $\alpha_j$, for each~$j$, gives
%an abacus representing $\tau$;
\item there exist unique partitions $\sigma^{(0)}, \ldots, \sigma^{(q)}$ satisfying~\eqref{eq:strips} and~\eqref{eq:rownumbers};
\item each skew partition $\rhonu(i)/\tau(i)$ in the $r$-quotient of $\rhonu/\tau$ 
has at most one box in each column of its Young diagram.
\end{thmlist}
\end{lemma}

\begin{proof}
Let $A$ be an abacus representing $\tau$. If $\beta$ is a position
in $A$ containing a bead then the row-number of the 
$r$-border strip corresponding to a single-step downward move of this bead
 is one more than the number of beads in the positions $\{ \beta+r+j : j \in \N \}$ of $A$.
Thus a sequence of single-step downward bead moves, moving beads in positions
$\beta_1, \ldots, \beta_{c}$ in that order, adds $r$-border strips in decreasing
order of their row number, as required by~\eqref{eq:rownumbers}, if and only if $\beta_1 \le \ldots \le \beta_c$.
It follows that (i) and (ii) are equivalent. It is easily seen that (ii) is equivalent to (iii) and~(iv).
\end{proof}

\subsubsection*{Ribbon tableaux}
Let $n \in \N_0$. Let $\rhonu/\tau$ be a skew partition of $rn$ and let~$\alpha$
be a composition of $n$ with exactly $\ell$ parts. 
An \emph{$r$-ribbon tableau} of \emph{shape} $\rhonu/\tau$ and \emph{weight} $\alpha$ is a sequence of partitions
\begin{equation} \tau = \rho^{(0)} \subset \rho^{(1)} \subset \ldots \subset \rho^{(\ell)} = \rhonu 
\label{eq:ribbons} \end{equation}
such that $\rho^{(j)}/\rho^{(j-1)}$ is a horizontal $r$-ribbon strip of size $r\alpha_j$ for each 
$j \in \{1,\ldots,\ell\}$.
We say that $\rho^{(j)}/\rho^{(j-1)}$ has \emph{label} $j$.
We denote the set of all  $r$-ribbon tableaux of shape $\rhonu/\tau$ and weight $\alpha$ by
$\rRT(\rhonu/\tau,\alpha)$. For an example see \S\ref{sec:example}~below.

%\subsubsection*{Multipartitions and multitableaux}

%
%\subsubsection*{Lattice words and tableaux}
%The \emph{content} of a word $w$ with entries in $\N$, denoted $\cont(w)$, is the composition $\alpha$
%such that, for each $k$, $\alpha_k$ is the number of positions of $w$ equal to $k$.
%For $\mu$ a partition, a word $w$ of length $\ell$
%with entries in $\N$ is a $\mu$-\emph{lattice word} if for every position %\emph{$\mu$-latticed} if for every position
%$i \in \{1,\ldots, \ell\}$ and all $k > 1$ we have
%\[ \bigl| \bigl\{ j : j \in \{i,\ldots, \ell\}, w_j = k-1 \bigr\} \bigr|  + \mu_{k-1}
%\ge \bigl| \bigl\{ j : j \in \{i,\ldots, \ell\}, w_j = k\bigr\}\bigr| + \mu_k.  \]
%(If $k-1$ or $k$ exceeds the number of parts of $\mu$ then take the corresponding part of $\mu$ to be zero.)
%If $t$ is a tableau then the \emph{word} of $t$ is obtained
%by reading the entries of $t$ from left to right in each row, starting in the
%greatest numbered row. We denote the word of $t$ by $\w(t)$ and define
%$\cont(t) = \cont(\w(t))$.

%\section{Combinatorial rules}

\subsection*{A plethystic Murnaghan--Nakayama rule}
In the second step of the proof of Theorem~\ref{thm:main} we need the following 
combinatorial rule.
Recall that~$h_\alpha$ denotes the complete symmetric function for the composition $\alpha$.

\begin{proposition}\label{prop:bs}
Let  $n \in \N_0$. If $\alpha$ is a composition of $n$ and~$\tau$ is a partition then
\[ s_\tau (h_\alpha \circ p_r) = \sum_{\nu} \bigl| \rRT(\nu/\tau,\alpha) \bigr| \sgn_r(\nu /\tau) s_\nu \]
%\sum_T \sgn_r(\nu / \tau) s_\nu
where the sum is over all partitions $\nu$ such that $\nu /\tau$ is a skew partition of~$rn$.
% $r$-ribbon tableaux $T$ of shape $\nu / \tau$ and weight $\alpha$.
\end{proposition}

This rule was first proved in
 \cite[page 29]{DesarmenienLeclercThibon}, using Muir's rule~\cite{Muir}. For an involutive proof of Muir's rule see \cite[Theorem~6.1]{LoehrAbacus}. 
The special case when $\tau = \varnothing$ and $\alpha$ has a single part is proved in \cite[I.8.7]{MacDonald}.
In this case  the result also follows from Chen's algorithm, as presented in \cite[page~130]{ChenGarsiaRemmel}.
The special case when~$\alpha$ has a single part was proved by the author
in~\cite{WildonPlethysticMN} using a sign-reversing involution. 
The general case then follows easily by induction, using that $h_{(\alpha_1,\ldots,\alpha_\ell)} 
\circ p_r = (h_{\alpha_1} \circ p_r)
\ldots (h_{\alpha_\ell} \circ p_r)$.

\subsection*{The Jacobi--Trudi formula}
\label{sec:rules}

Let $\ell \in \N$.
The symmetric group $\Sym_\ell$ acts on $\Z^\ell$ by place permutation.
Given $\alpha \in \Z^\ell$ and $g \in \Sym_\ell$, we define
$g \cdot \alpha = g (\alpha + \rho) - \rho$
where $\rho = (\ell-1,\ldots,1,0)$. For later use we note that
if $k \in \{1,\ldots,\ell-1\}$~then %if $g$ is the transposition $(i,i+1)$ then
\begin{equation}\label{eq:dotswap}
(k,k+1) \cdot \alpha = (\alpha_1, \ldots, \alpha_{k+1}-1,\alpha_k+1,\ldots, \alpha_\ell)
\end{equation}
where the entries in the middle are in positions $k$ and $k+1$.

The Jacobi--Trudi formula states that if $\lambda$ is a partition
with exactly $\ell$ parts and $\lambda / \mu$ is a skew partition then
\[ s_{\lambda/\mu} = \sum_{g \in \Sym_\ell} \sgn(g) h_{g \cdot \lambda - \mu}, \]
where if $\alpha$
has a strictly negative entry then we set $h_\alpha = 0$.
A proof of the formula is given in \cite[page 342]{StanleyII} by a beautiful involution
on certain tuples of  paths in $\Z^2$.

\section{A generalized Lascoux--Sch{\"u}tzenberger involution}
\label{sec:LS}

%We present a small generalization of the Lascoux--Sch{\"u}tzenberger involution used in \cite{RemmelShimozono}.
We begin by presenting the coplactic maps in \cite[\S 5.5]{Lothaire}. For further background
see \cite{LascouxSchutzenberger}.
Let $w$ be a word with entries in $\N$
and let $k \in \N$. Following the exposition in \cite{RemmelShimozono},
we replace each $k$ in $w$ with a right-parenthesis~`)' 
and each $k+1$ with a left-parenthesis~`('.
An entry $k$ or $k+1$ is $k$-\emph{paired} if its parenthesis has a pair,
according to the usual rules of bracketing, and otherwise $k$-\emph{unpaired}. 
Equivalently, reading $w$ from left to right, 
an entry~$k$ is $k$-unpaired if and only if it sets
a new record for the excess of $k$s over $(k+1)$s; dually, reading from right to left,
an entry $k+1$ is $k$-unpaired if and only if it sets a new record for the excess of $(k+1)$s over $k$s.
We may omit the `$k$-'
if it will be clear from the context.

\renewcommand{\underline}{\mathbf}

For example, if $w = 342\underline{22}4\underline{3}312\underline{3}11$ then the 
$2$-unpaired entries are shown in bold
and the corresponding
parenthesised word
is $(4)\mbf{))}4\mbf{(}(1)\mbf{(}11$ 

%Reading $w$ from left to right, his observation proves the following lemma.

\begin{lemma}\label{lemma:unpairedsubword}
Let $w$ be a word with entries in $\N$. Let $k \in \N$.
The subword of $w$ formed from its $k$-unpaired entries is
%\begin{equation}\label{eq:unpaired} 
$k^c (k+1)^d$ %\end{equation}
 for some $c$, $d \in \N_0$. Changing this subword to $k^{c'}\negvthinspace (k+1)^{d'}$, where $c'$, $d'\ \in \N_0$ 
 and $c'+d'=c+d$,
 while keeping all other positions the same, gives
a new word which has $k$-unpaired entries in exactly the same positions as~$w$.\hfill$\qed$
\end{lemma}

\begin{proof}
It is clear that any $k$ to the right of the rightmost unpaired $k+1$ in $w$ is paired.
Dually,
any $k+1$ to the left of the leftmost unpaired $k$ in~$w$ is paired. Hence
the subword of $w$ formed from its unpaired entries has the claimed form. When $d \ge 1$, changing
the unpaired subword from $k^c (k+1)^d$ to $k^{c+1}(k+1)^{d-1}$ replaces the first unpaired $k+1$,
in position $i$ say, with a $k$; 
since every $k+1$ to the left of position $i$ is paired, the new $k$ is unpaired. The dual
result holds when $c \ge 1$; together these imply the lemma.
\end{proof}

\begin{definition}\label{def:wordmaps}
Let $w$ be a word with entries from $\N$.
Suppose that the $k$-unpaired subword of $w$ is $k^c(k+1)^d$.
If $d > 0$, let $E_k(w)$ be defined by changing
the subword to $k^{c+1}(k+1)^{d-1}$, and if $c > 0$, let
$F_k(w)$ be defined by changing the subword to $k^{c-1}(k+1)^{d+1}$.
Let $S_k(w)$ be defined by changing the subword to $k^d(k+1)^c$.
\end{definition}

We now extend these maps to tuples of skew tableaux.
Let $\cont(t)$ denote the content
of a skew tableau $t$, and let $\w(t)$ denote its word, obtained by reading the rows of $t$ from left to right,
starting at the highest numbered row.
Let $m \in \N$ and let $\sigmab/\taub = 
\bigl(\sigma(1)/\tau(1), \ldots, \sigma(m)/\tau(m)\bigr)$ be a skew $m$-multipartition of $n \in \N$.
Let $\ell \in \N$ and let $\alpha \in \Z^\ell$.
Let $\mSSYT(\sigmab/\taub,\alpha)$ denote the set of all $m$-tuples $\bigl( t(1),\ldots,t(m)\bigr)$
of semistandard skew tableaux such that~$t(i)$ has shape $\sigma(i)/\tau(i)$ for each 
$i \in \{1,\ldots,m\}$ and
\begin{equation}\label{eq:cont} \cont\bigl(t(1)\bigr) + \cdots + \cont\bigl(t(m)\bigr) = \alpha. \end{equation}
Thus if $\alpha$ fails to be a composition because it has a negative entry then
$\mSSYT(\sigmab/\taub,\alpha) = \varnothing$. We call the elements
of $\mSSYT(\sigmab/\taub, \alpha)$
\emph{semistandard skew
$m$-multitableaux}
of \emph{shape} $\sigmab/\taub$, or \emph{$m$-multitableaux} for short.
The \emph{word} of an $m$-multitableau $\bigl(t(1),\ldots,t(m)\bigr)$ $\in \mSSYT(\sigmab/\taub,\alpha)$
is the concatenation
$\w\bigl(t(1)\bigr) \ldots \w\bigl(t(m)\bigr)$. For $k \in \N$  
we say that an entry of an $m$-multitableau~$\mathbf{t}$ is $k$-\emph{paired} if the corresponding
entry of $\w(\mathbf{t})$ is $k$-paired.
Note that, for fixed $\sigmab/\taub$, a word $w$ of length $n$ and content~$\alpha$ uniquely
determines an $m$-multitableau of shape $\sigmab/\taub$ satisfying~\eqref{eq:cont};
we denote this multitableau by $\T(w)$. (The skew $m$-multipartition $\sigmab/\taub$ will always be clear from the context.) 
Abusing notation slightly, we set
$E_k(\mathbf{t}) = \T\bigl( E_k(\w(\mathbf{t}))$, $F_k(\mathbf{t}) = \T\bigl( F_k(\w(\mathbf{t}))$  
(when either is defined) and 
$S_k(\mathbf{t}) = \T\bigl( S_k(\w(\mathbf{t}))$.

\begin{example}\label{ex:invol}
Consider the semistandard skew $3$-multitableau 
\[ \mathbf{t} = \left( \, \young(2\bz\bz,34)\,, \ \young(:2\bd,\bd3,4)\,, \ \young(11)\, \right).\]
The shape of $\mathbf{t}$ is $\bigl( (3,2), (3,2,1) /(1), (2) \bigr)$
and the $2$-unpaired entries are shown in bold. By Definition~\ref{def:wordmaps},~$E_2(\mathbf{t})$ is obtained from~$\mathbf{t}$
by changing the leftmost unpaired $3$ to a $2$, and $F_2(\mathbf{t})$ is obtained from $\mathbf{t}$
by changing the rightmost unpaired $2$ to a $3$. It follows that
\[ S_2E_2(\mathbf{t}) = 
\left( \, \young(2\bz\bd,34)\,, \ \young(:2\bd,\bd3,4)\,, \ \young(11)\, \right). 
\]
\end{example}

As mentioned in the introduction, one may identify a skew $m$-multitableau with
a single skew tableau of larger shape. For example, 
the semistandard skew $3$-multitableau $\mathbf{t}$ above corresponds to
\[ \scalebox{1}{$\young(::::::11,::::2\bd,:::\bd3,:::4,2\bz\bz,34)$\, .} \]
This identification may be used to reduce the next two results to 
Proposition~4 and the argument in \S 3 of \cite{RemmelShimozono}. We avoid it
in this paper, since it has an artificial flavour, and loses  combinatorial data: for instance,
the skew tableau above may also be identified with two different semistandard skew $2$-multitableaux.

\begin{lemma}\label{lemma:coplactic}
Let $m \in \N$,
let $\sigmab/\taub$ be a skew $m$-multipartition of $n \in \N_0$ 
and let $\alpha$ be a composition with exactly $\ell$ parts.
Fix $k \in \{1,\ldots,\ell-1\}$. Let $\mSSYT_k(\sigmab/\taub,\alpha)$ and 
$\mSSYT_{k+1}(\sigmab/\taub,\alpha)$ 
be the sets of $m$-multitableaux in $\mSSYT(\sigmab/\taub,\alpha)$
that have a $k$-unpaired~$k$ or a $k$-unpaired $k+1$, respectively.
%and let $\mSSYT_{k+1}(\sigmab/\taub,\alpha)$ be the set of all multitableaux having a $k$-unpaired~$(k+1)$
Let 
\[ \alphaepsilon(k) = (0,\ldots, 1, -1,\ldots,0) \in \Z^\ell,\] 
where the two non-zero entries are in positions $k$ and $k+1$.
The maps
\begin{align*}
%\text{\emph{(i)}}\ 
E_k &: \mSSYT_{k+1}(\sigmab/\taub,\alpha) \rightarrow \mSSYT_k\bigl( \sigmab/\taub,\alpha+\alphaepsilon(k) \bigr) \\
%\text{\emph{(ii)}}\ 
F_k &: \mSSYT_k(\sigmab/\taub,\alpha) \rightarrow \mSSYT_{k+1}\bigl( \sigmab/\taub,\alpha-\alphaepsilon(k) \bigr) \\
%\text{\emph{(iii)}}\ 
S_k &: \mSSYT_k(\sigmab/\taub,\alpha) \rightarrow \mSSYT_{k+1}\bigl( \sigmab/\taub,(k,k+1)\alpha \bigr)
\end{align*}
are bijections and $S_kE_k : \mSSYT_{k+1}(\sigmab/\taub,\alpha) \rightarrow
\mSSYT_{k+1}(\sigmab/\taub, (k,k+1) \cdot \alpha)$ is an involution.
\end{lemma}

\begin{proof}
Let $\mathbf{t} = \bigl(t(1),\ldots,t(m)\bigr) \in \mSSYT(\sigmab/\taub,\alpha)$. 
The main  work comes in showing that
$E_k(\mathbf{t})$, $F_k(\mathbf{t})$  are semistandard (when defined).
Suppose that $E_k(\mathbf{t}) = \bigl(t'(1),\ldots,t'(m)\bigr)$ and that
the first %$k$-
unpaired $k+1$ in $\w(\mathbf{t})$
corresponds to the entry in row $a$ and column~$b$ of tableau $t(j)$. 
Thus $t'(j)$
is obtained from $t(j)$ by changing this entry to an unpaired~$k$
and $t'(i) = t(i)$ if $i\not=j$.

Let $t = t(i)$, let $t' = t'(i)$ and
write $u_{(a,b)}$ for the entry of a tableau~$u$ in row $a$ and column $b$.
If $t'$ fails to be semistandard then $a > 1$, $(a-1,b)$ is a box in $t$, and 
$t'_{(a-1,b)} = k$. Hence $t_{(a-1,b)} = k$. This~$k$ is
to the right of the unpaired $k+1$ in $\w(t)$, so 
by~Lemma~\ref{lemma:unpairedsubword} it is paired, necessarily
with a $k+1$ in row $a$ and some column $b' > b$ of $t$.
Since 
\[ k = t_{(a-1,b)} \le t_{(a-1,b')} < t_{(a,b')} = k+1 \] 
we have $t_{(a-1,b')} = k$.
Thus $t_{(a,e)} = k+1$ and $t_{(a-1,e)} = k$ for every $e \in \{b,\ldots, b'\}$.
Since $t_{(a-1,b)}$ is paired with $t_{(a,b')}$ under the $k$-pairing, we see that $t_{(a-1,b+j)}$ is paired with
 $t_{(a,b'-j)}$
for each $j \in \{0,\ldots, b'-b\}$. %\le b'-b$. 
In particular, the $k+1$ in position $(a,b)$ of $t$ is paired,
a contradiction. Hence $E_k( \mathbf{t})$ is semistandard. %t(1),\ldots,t(m)\bigr)$ is semistandard. 
The proof is similar 
for $F_k$ in the case when $\mathbf{t}$ has an unpaired~$k$.

It is now routine to check that $E_kF_k$ and $F_kE_k$ are the identity maps on their respective domains, so $E_k$ and $F_k$
are bijective. If the unpaired subword of $\w(\mathbf{t})$ is 
$k^c (k+1)^d$ then $S_k(\mathbf{t}) = E_k^{d-c}(\mathbf{t})$ 
if $d \ge c$ and $S_k(\mathbf{t}) = F_k^{c-d}$ if $c \ge d$. Hence $S_k$ is an involution.
A similar argument shows that $S_k E_k$
is  an involution. By~\eqref{eq:dotswap} at the end of \S\ref{sec:background},
the image of $S_kE_k$ is as claimed.
\end{proof}

We are  ready to define our key involution. Say that a semistandard skew multitableau $\mathbf{t}$
is \emph{latticed} if $\w(\mathbf{t})$ has no $k$-unpaired $(k+1)$s, for any $k$.
Let $\lambda$ be a partition of $n \in \N_0$ with exactly~$\ell$ parts, let~$\sigmab/\taub$
be a skew $m$-multipartition of $n$ and let 
\begin{equation}\label{eq:T} 
\mathcal{T} = \bigcup_{g \in \Sym_\ell} \mSSYT(\sigmab/\taub, g \cdot \lambda). \end{equation}
Observe that if $g\not= \id_{\Sym_\ell}$ then $g \cdot \lambda$ is not a partition, and so
no element of $\mSSYT(\sigmab/\taub, g \cdot \lambda)$ is latticed. Therefore the set
\[ \mSSYTL(\sigmab/\taub, \lambda) = \bigl\{ \mathbf{t} \in \mSSYT(\sigmab/\taub, \lambda) : 
\text{$\mathbf{t}$ is latticed} \bigr\} \]
is precisely the latticed elements of $\mathcal{T}$.
Let $\mathbf{t} \in \mathcal{T}$. If $\mathbf{t}$ is latticed then define $G(\mathbf{t}) = \mathbf{t}$.
Otherwise consider the $k$-unpaired entries in $\w(\mathbf{t})$ for each $k \in \N$. If
the rightmost $k$-unpaired entry (for some $k$) is $k+1$ then 
define $G(\mathbf{t}) = S_kE_k(\mathbf{t})$. 

For instance, in Example~\ref{ex:invol} we have $k=2$
and $G(\mathbf{t}) = S_2E_2(\mathbf{t})$.

\begin{proposition}\label{prop:LS}
Let $m \in \N$,
let $\sigmab/\taub$ be a skew $m$-multipartition of $n \in \N_0$, and let $\lambda$ be a partition
of $n$. Let $\mathcal{T}$ be as defined in~\eqref{eq:T}.
The map $G : \mathcal{T} \rightarrow \mathcal{T}$ is an involution fixing precisely the skew
$m$-multitableaux
in $\mSSYTL(\sigmab/\taub,\lambda)$. If $\mathbf{t} \in \mSSYT(\sigmab/\taub, g\cdot \lambda)$
and $G(\mathbf{t}) \not= \mathbf{t}$ then $G(\mathbf{t}) \in \mSSYT(\sigmab/\taub, (k,k+1)g \cdot \lambda)$
for some $k \in \{1,\ldots, \ell-1\}$.
\end{proposition}

\begin{proof}
This follows immediately from Lemma~\ref{lemma:coplactic}.
\end{proof}

This is a convenient place to define our generalized Littlewood--Richardson coefficients. 
In \S\ref{subsec:nonPlethystic} we show these specialize to the original definition.

\begin{definition}
The \emph{Littlewood--Richardson coefficient} corresponding to a partition $\lambda$ of $n$ 
and a skew $m$-multipartition $\sigmab/\taub$ of $n$ is
\[ c^{\lambda}_{\sigmab/\taub} = \SSYTL(\sigmab/\taub,\lambda). \]
%$\bigl(\sigma(1)/\tau(1),\ldots,\sigma(m)/\tau(m)\bigr)$ of $n$ is
%%\[ %c^\lambda_\sigmab = 
%$c^{\lambda}_{(\sigma(1)/\tau(1),\ldots,\sigma(m)/\tau(m))} = 
%%\bigl| \mSSYTL( \sigmab, \lambda) \bigr|. \]
%\bigl|
%\mSSYTL\bigl( \bigl(\sigma(1)/\tau(1),\ldots,\sigma(m)/\tau(m)\bigr), \lambda \bigr)
%\bigr|$. %\]
\end{definition}

\vspace*{-3pt}

\enlargethispage{3pt}

\section{Proof of Theorem~\ref{thm:main}}
\label{sec:proof}

Suppose that $\lambda$ has exactly $\ell$ parts.
The outline of the proof is as follows:
\begin{align}
s&_\tau(s_{\lambda / \mu} \circ p_r) \\
	&= \sum_{g \in \Sym_\ell} \!\sgn(g) s_\tau (h_{g \cdot \lambda - \mu} \circ p_r) \label{eq:JacobiTrudi} \\
	&= \sum_{g \in \Sym_\ell } \!\sgn(g) \sum_{\nu} \bigl|  \rRT(\nu/\tau, g \cdot \lambda - \mu) 
	\hskip-0.5pt
	\bigr| \sgn_r(\nu/\tau) s_\nu \label{eq:RT} \\
	&= \sum_{g \in \Sym_\ell} \!\sgn(g) \sum_{\nub} \bigl|  \mSSYT(\nub/\taub, g \cdot \lambda - \mu) 
	\hskip-0.5pt
	\bigr| \sgn_r\bigl( (\nub/\taub, \tau)^\star \bigr) s_{(\nub/\taub, \tau)^\star} \label{eq:mSSYT} \\
 	&= \sum_{\nub} \bigl|  \mSSYTL \bigl( \nub /\taub : \mu,  \lambda \bigr) \bigr|
	  \sgn_r\bigl( (\nub/\taub, \tau)^\star \bigr) s_{(\nub/\taub, \tau)^\star} \label{eq:afterInvol} \\
	&= \sum_{\nub} c^\lambda_{\nub/\taub \vthinspace : \vthinspace \mu}
	%c^\lambda_{(\nu(0)/\tau(0), \ldots, \nu(r-1)/\tau(r-1), \mu) } 
	\sgn_r\bigl( (\nub/\taub, \tau)^\star \bigr) s_{(\nub/\taub, \tau)^\star}, \label{eq:ident}
\end{align}
where the sum in~\eqref{eq:RT} is over all partitions $\nu$ such that $\nu/\tau$ is a skew partition
of $rn$, the sums in~\eqref{eq:mSSYT} and~\eqref{eq:afterInvol} are over all $r$-multipartitions
$\nub$ such that $\nub / \taub$ is a skew $r$-multipartition of $n$, and in~\eqref{eq:afterInvol}
and~\eqref{eq:ident},
$\nub/\taub : \mu$ is the skew $(r+1)$-multipartition $\bigl(\nu(0)/\tau(0),\ldots,\nu(r-1)/\tau(r-1),
\mu\bigr)$
obtained from~$\nub/\taub$ by appending~$\mu$.

%In~\eqref{eq:afterInvol}, $\mSSYTL_\mu(\nu,\lambda)$ is, by definition, the 
%set of all $\bigl( T(0), \ldots, T(r-1)\bigr) \in \mSSYT(\nub, \lambda)$ such that
%$\bigl( T(0), \ldots, T(r-1), S\bigr)$ is latticed, where $S$ is the semistandard $\mu$-tableau
%having all its entries in its $j$-th row equal to $j$.

We now give an explicit bijection or involution establishing each step. For an illustrative
example see \S\ref{sec:example} below.

\begin{proof}[Proof of \eqref{eq:JacobiTrudi}] Apply the Jacobi--Trudi formula for skew Schur functions,
as stated in \S\ref{sec:rules}.
%to $s_{\lambda / \mu}$. 
\end{proof}

\begin{proof}[Proof of \eqref{eq:RT}] Apply Proposition~\ref{prop:bs} to each
$s_\tau(h_{g \cdot \lambda - \mu} \circ p_r)$.
\end{proof}

\begin{proof}[Proof of \eqref{eq:mSSYT}] Let $T$ be a 
$r$-ribbon tableau
of shape~$\nu/\tau$ and weight $\alpha$ as in~\eqref{eq:ribbons},
so $T$ corresponds to the sequence of partitions
\[ \tau = \rho^{(0)} \subset \rho^{(1)} \subset \ldots \subset \rho^{(\ell)} = \nu 
%\label{eq:ribbons}
\]
where $\rho^{(j)}/\rho^{(j-1)}$ is a horizontal $r$-ribbon strip of size $r\alpha_j$ for each $j
\in \{1,\ldots,\ell\}$.
Let $\nu/\tau$ have $r$-quotient $\nub/\taub = 
\bigl(\nu(0)/\tau(0), \ldots, \nu(r-1)/\tau(r-1)\bigr)$, so $(\nub/\taub, \tau)^\star = \nu$.
Take an abacus $A$ representing $\tau$ with a multiple of $r$ beads.
The sequence above
defines a sequence of single-step downward bead moves leading from $A$ to an abacus
$B$ representing~$\nu$. For each bead moved on runner $i$ put the label of the corresponding
horizontal $r$-ribbon strip in the
corresponding box of the Young diagram of $\nu(i)/\tau(i)$.
By Lemma~\ref{lemma:ribboneqv}(iv), this defines a semistandard skew tableau $t(i)$ of shape $\nu(i)/\tau(i)$ 
for each~$i \in \{0,\ldots,r-1\}$.
Conversely, given $\bigl( t(0), \ldots, t(r-1)\bigr) \in \mSSYT(\nub/\taub, \alpha)$, 
one obtains a sequence
of single-step downward bead moves satisfying the condition in Lemma~\ref{lemma:ribboneqv}(ii), 
and hence an $r$-ribbon tableau of shape $\nu/\tau$ and content $\alpha$.
Thus the map sending $T$ to $\bigl( t(0), \ldots, t(r-1)\bigr)$ is a bijection
from $\rRT(\nu/\tau, g \cdot \lambda - \mu)$
to $\mSSYT(\nub/\taub, g \cdot \lambda - \mu)$, as required. 
\end{proof}

\begin{proof}[Proof of \eqref{eq:afterInvol}]
Fix a skew $r$-multipartition $\nub/\taub$ of $n$. Let
\[ \mathcal{T} = \bigcup_{g \in \Sym_\ell} \mSSYT(\nub/\taub : \mu, g \cdot \lambda). \]
%where $\nub : \mu$ denotes the multipartition $\bigl( \nu(0), \ldots, \nu(r-1), \mu \bigr)$.
Let $G$ be the involution on $\mathcal{T}$ defined in \S\ref{sec:LS}.
Let $u(\mu)$ be the semistandard $\mu$-tableau having all its entries in its $j$-th row equal to $j$
for each relevant $j$.
Note that $u(\mu)$ is the unique  latticed semistandard  $\mu$-tableau. Thus if
\begin{equation}
\label{eq:finalFixed} \mathcal{T}_\mu = \bigl\{ \bigl(t(0), \ldots, t(r-1), v \bigr) \in \mathcal{T} :  
v = u(\mu)\bigr\} \end{equation}
then $\mSSYTL(\nub/\taub : \mu, \lambda) \subseteq \mathcal{T}_\mu$.
Let $\mathbf{t} \in \mathcal{T}_\mu$. %If $\mu$ is a partition of $s$ then the
The final $|\mu|$ positions of $\w(\mathbf{t})$
correspond to the entries of $u(\mu)$. 
%final $|\mu|$ positions
%of $\w(\mathbf{t})$.
Every entry $k + 1$ in these positions is $k$-paired. If an entry $k$ in one of these
positions is $k$-unpaired then there is no $k$-unpaired $k+1$ to its left, so 
every $k+1$ in $\w(\mathbf{t})$ is $k$-paired. 
It follows that the final semistandard tableau in~$G(\mathbf{t})$ is $u(\mu)$
and so~$G$ restricts to an involution on~$\mathcal{T}_\mu$. 
By Proposition~\ref{prop:LS}, the fixed-point set of~$G$, acting on either
$\mathcal{T}$ or~$\mathcal{T}_\mu$,
is $\mSSYTL(\nub/\taub : \mu, \lambda)$.

%
%Moreover, if $(T(0), \ldots, T(r-1), U) \in \mSSYTL(\nub : \mu, \lambda)$
%then, since every entry $k > 1$ of $\w(U)$ is paired, $U = S$. 
%Therefore, by Proposition~\ref{prop:LS}, the fixed point set of $G$ on either $\mathcal{T}$ or $\mathcal{T}'$
%is $\mSSYTL(\nub : \mu, \lambda)$.

%There is an obvious bijection
%\[ \bigcup_{g \in \Sym_\ell} \mSSYT(\nub, g \cdot \lambda - \mu) \rightarrow \mathcal{T}'. \]
%%\xrightarrow{\text{append $S$}}\mathcal{T}' \]
%given by appending $S$ to each $r$-multitableau. 

The part of the sum in~\eqref{eq:mSSYT} corresponding to the skew $r$-multipartition $\nub/\taub$ is
\[ \sum_{g \in \Sym_\ell}\, \sum_{\mathbf{t}\in \SSYT(\nub/\taub, g \cdot \lambda - \mu)} \sgn(g)\sgn_r\bigl( (\nub/\taub, \tau)^\star \bigr)  s_{ (\nub/\taub, \tau)^\star}. \]
%where the third sum is over all $\mathbf{t} \in \SSYT(\nub, g \cdot \lambda - \mu)$.
The set of $r$-multitableaux $\mathbf{t}$ %of shape $\nub$ 
 in this sum is
$\mathcal{S} = \bigcup_{g \in \Sym_\ell} \mSSYT(\nub/\taub, g \cdot \lambda -\mu)$.
There is an obvious bijection $A : \mathcal{S} \rightarrow \mathcal{T}_\mu$
given by appending $u(\mu)$ to a skew $r$-multitableau in $\mathcal{S}$. 
By the remarks above,~$A^{-1} G A$ is an
involution on~$\mathcal{S}$.
Since  $\sgn(g) = - \sgn\bigl( (k,k+1)g \bigr)$, it follows from Proposition~\ref{prop:LS} that
the contributions to~\eqref{eq:mSSYT}
from $r$-multitableaux $\mathbf{t} \in \mathcal{S}$ such that $A(\mathbf{t}) \not\in 
\mSSYTL(\nub/\taub : \mu, \lambda)$
cancel in pairs, leaving exactly the  $r$-multitableaux $\mathbf{t}$ such 
that $A(\mathbf{t}) \in \mSSYTL(\nub/\taub : \mu, \lambda)$. This proves~\eqref{eq:afterInvol}.
\end{proof}

\begin{proof}[Proof of \eqref{eq:ident}]
This is true by our definition of the Littlewood--Richardson coefficient 
$c^\lambda_{\nub/\taub \vthinspace:\vthinspace \mu}$. %(\nu(0)/\tau(0),\ldots,\nu(r-1)/\tau(r-1)\mu)}$.
\end{proof}

\bigskip
\section{Example}
\label{sec:example}

We illustrate~\eqref{eq:mSSYT} and \eqref{eq:afterInvol} in the proof of Theorem~\ref{thm:main}.
Let $r=3$, let $\lambda = (3,3)$, $\mu = \varnothing$ and $\tau = (3,2)$.
Take $\nu = (6,5,5,5,2)$. From the abaci
shown in Figure~2 below,
we see that $\nub /\taub= \bigl( (1), (2), (2,2) / (1) \bigr)$.
\begin{figure}[b]
\begin{center}
\includegraphics{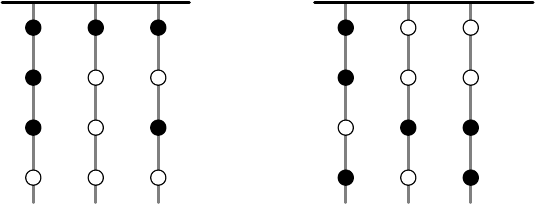}
\end{center}
\caption{Abaci for $(3,2)$ and $(6,5,5,5,2)$.}
\end{figure}
We have $(1,2)\cdot \lambda = (2,4)$. The horizontal $3$-ribbon
tableaux of shape $(6,5,5,5,2)$ and weights $(3,3)$ and $(2,4)$
are shown in Figure~3 above.
\begin{figure}[t]
\hspace*{-0.25in}\includegraphics{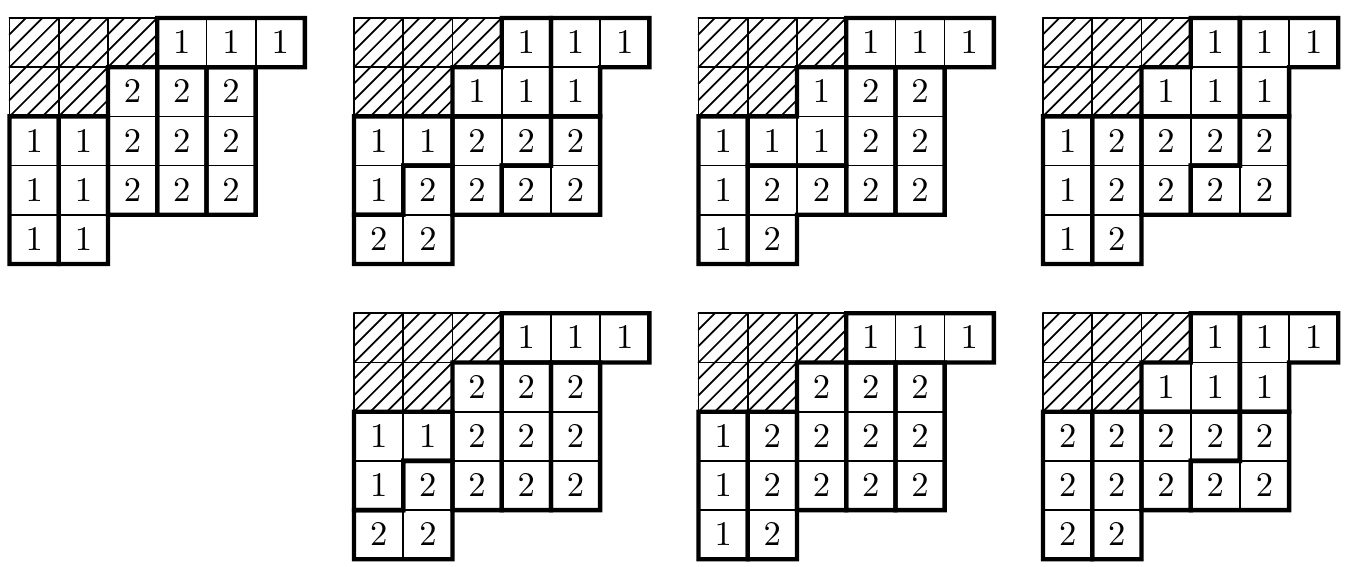}
%\hspace*{-0.75in}\includegraphics{Fig2RTs3351.pdf} 
\caption{\small The four $3$-ribbon tableau tableaux in $3$-$\RT\bigl( (6,5,5,5,2)/(3,2),(3,3) \bigr)$ 
and the corresponding $3$-ribbon tableaux in $3$-$\RT\bigl( (6,5,5,5,2)/(3,2), (2,4)\bigr)$ 
under the $G$ involution. The first tableaux in the top line is latticed, and so is fixed by~$G$.
In each case the ribbon with label $i$ is marked $i$, and its unique partition into $3$-border strips,
as in~\eqref{eq:strips},
is shown by heavy lines.}
\end{figure}
%We have $\nub = \bigl( (2), (1), (2,1) \bigr)$.
Applying the  bijection
%$ \text{$3$-$\RT$}\bigl( (6,5,5,5,2)/(3,2),$ $(3,3) \bigr) \rightarrow \mSSYT\bigl( 
%\bigl( (1), (2), (2,1) / (1) \bigr), (3,3) \bigr)$
\[ \text{$3$-$\RT$}\bigl( (6,5,5,5,2)/(3,2), (3,3) \bigr) \rightarrow \mSSYT\bigl( 
\bigl( (1), (2), (2,2) / (1) \bigr), (3,3) \bigr) \]
in the proof of~\eqref{eq:mSSYT} we obtain the $3$-multitableaux 
%$\mathbf{T}_1$,
%$\mathbf{T}_2$, $\mathbf{T}_3$, $\mathbf{T}_4$,
\begin{align}
\mathbf{t}_1 = \Bigl(\, \young(2)\,,\ \young(12)\,,  \ \young(:1,12)\, \Bigr),&\quad
\mathbf{t}_2 =\Bigl(\, \young(1)\,,\ \young(22)\,, \ \young(:1,12)\, \Bigr),  \notag \\
\mathbf{t}_3 = \Bigl(\, \young(2)\,,\ \young(11)\,, \ \young(:1,22)\, \Bigr),&\quad
\mathbf{t}_4 = \Bigl(\, \young(1)\,,\ \young(12)\,, \ \young(:1,22)\, \Bigr) \label{eq:t1234}
\end{align}
in the order corresponding to the top line in Figure~3. Here $\mathbf{t}_1$ is latticed
and $\mathbf{t}_2$, $\mathbf{t}_3$, $\mathbf{t}_4$ are not. 
 Applying the involution $G$ in~\S\ref{sec:LS} to $\mathbf{t}_2$,
$\mathbf{t}_3$, $\mathbf{t}_4$,
as in the proof of~\eqref{eq:afterInvol}, we obtain the $3$-multitableaux
\[ \Bigl(\, \young(2)\,,\ \young(22)\,, \ \young(:1,12)\, \Bigr), 
\Bigl(\, \young(2)\,,\ \young(12)\,, \ \young(:1,22)\, \Bigr),
\Bigl(\, \young(1)\,,\ \young(22)\,, \ \young(:1,22)\, \Bigr) \]
in the order corresponding to the bottom line in Figure~3.
As expected, these are the images of the three horizontal $3$-ribbon tableaux 
of shape $(6,5,5,5,2)$ and weight $(2,4)$ under the bijection
\[ \text{$3$-$\RT$}\bigl( (6,5,5,5,2)/(3,2), (2,4) \bigr) \rightarrow \mSSYT\bigl( 
\bigl( (1), (2), (2,1) / (1) \bigr), (2,4) \bigr).\]
Therefore all but one of the seven summands in~\eqref{eq:mSSYT} is cancelled by~$G$. Since
$\sgn_3\bigl( (6,5,5,5,2) / (3,2) \bigr) = 1$, we have 
$\langle s_{(3,2)}(s_{(3,3)} \circ p_3), s_{(6,5,5,5,2)} \rangle = 1$.

We now find $\langle s_{(3,2)}(s_{(4,3)/(1)} \circ p_3), s_{(6,5,5,5,2)} \rangle$
using the full generality of Theorem~\ref{thm:main}.
Following the proof of~\eqref{eq:afterInvol}, we append% the unique
%semistandard tableau of shape $(1)$, namely 
$\;\young(1)\,$ to each of the four $3$-multitableaux in
$\mSSYT\bigl( 
\bigl( (1), (2), (2,1) / (1) \bigr), (3,3) \bigr)$ before applying~$G$.
This gives three latticed $4$-multitableaux, 
\[  \hspace*{-0.1in}\Bigl(\, \young(2)\,,\ \young(12)\,,  \ \young(:1,12)\,, \ \young(1)\,\Bigr),
\Bigl(\, \young(1)\,,\ \young(22)\,, \ \young(:1,12)\,, \ \young(1)\,\Bigr), 
\Bigl(\, \young(2)\,,\ \young(11)\,, \ \young(:1,22)\,, \ \young(1)\, \Bigr), \]
all fixed by $G$, and one unlatticed $4$-multitableau, obtained by appending~$\young(1)$ to $\mathbf{t}_4$;
its image under $G$ is given by
\[ 
\Bigl(\, \young(1)\,,\ \young(12)\,, \ \young(:1,22)\, , \ \young(1)\, 
\Bigr)  \stackrel{G}{\longleftrightarrow}
\Bigl(\, \young(2)\,,\ \young(22)\,, \ \young(:1,22)\,, \ \young(1)\,  \Bigr).
  \]
There are now five summands in~\eqref{eq:mSSYT}, of which two are cancelled by $G$,
and so $\langle s_{(3,2)}(s_{(4,3)/(1)} \circ p_3), s_{(6,5,5,5,2)} \rangle = 3$.
Alternatively, we can get the same result by using
(a very special case of) the Littlewood--Richardson
rule to write $s_{(4,3) / (1)} = s_{(4,2)} + s_{(3,3)}$. From above 
we have $\langle s_{(3,2)}(s_{(3,3)} \circ p_3), s_{(6,5,5,5,2)} \rangle = 1$ and
since $h_{(2,4)} = h_{(4,2)}$, we have
\[  \langle
s_{(3,2)}(h_{(4,2)} \circ p_3), s_{(6,5,5,5,2)} \rangle 
= \bigl| \text{$3$-RT}\bigl( (6,5,5,5,2), (2,4) \bigr) \bigr| 
= 3. \] 
Since %$(1,2) \cdot (4,2) = (1,5)$ and
$\bigl| \text{$3$-RT}\bigl( (6,5,5,5,2), (1,5) \bigr) \bigr| = 1$, we get
$\langle s_{(3,2)}(s_{(4,3)/ (1)} \circ p_3), s_{(6,5,5,5,2)} \rangle = (4-3)+(3-1) = 3$, as before.
This extra cancellation 
suggests that the general form of the SXP rule in Theorem~\ref{thm:main} may have some computational advantages.

\section{Connections with other combinatorial rules} 
\label{sec:SXPp}

\subsection{Non-plethystic rules}\label{subsec:nonPlethystic}

Let $\SSYT(\nu/\tau, \lambda)$ be the set of semistandard skew
 tableaux of shape $\nu/\tau$ and content~$\lambda$. 
We say  that a skew tableau~$t$ is \emph{latticed} if the corresponding skew $1$-multitableau $(t)$
is latticed. Let $\SSYTL(\nu/\tau,\lambda)$ be the set of latticed semistandard
tableaux of shape $\nu/\tau$ and content~$\lambda$.  

Let $\lambda / \mu$ be a skew partition of $n \in \N_0$.
Setting $r=1$ in Theorem~\ref{thm:main} 
we obtain 
\begin{equation} 
s_\tau s_{\lambda / \mu} = \sum_\nu c^{\lambda}_{(\nu/\tau\vthinspace,\thinspace \mu)} s_\nu \label{eq:lrGen} \end{equation}
where the sum is over all
 partitions $\nu$ such that $\nu/\tau$ is a skew partition of~$n$.
(For the remainder of this subsection we usually rely on the context to make such summations clear.)
%This is equivalent to a special case of the skew-skew Littlewood--Richardson
%rule in \cite[\S 4]{RemmelShimozono}. 
Specialising~\eqref{eq:lrGen} 
further by setting $\mu = \varnothing$ we get
\begin{equation}
\label{eq:lrProd}
s_\tau s_\lambda = \sum_\nu c^{\lambda}_{(\nu/\tau)} s_\nu. % = \sum_\nu |\SSYTL(\nu/\tau, \lambda)| s_\nu 
\end{equation}
By definition $c^{\lambda}_{(\nu/\tau)}
= |\SSYTL(\nu/\tau,\lambda)|$. Thus~\eqref{eq:lrProd} is the original
Littlewood--Richardson rule, as proved in \cite[Theorem~III]{LittlewoodRichardson}.

Specialising~\eqref{eq:lrGen} in a different way by setting $\tau = \varnothing$, and then
changing notation for consistency with~\eqref{eq:lrProd}, we get
\begin{equation}\label{eq:lrSkew} s_{\nu/\tau} = \sum_{\lambda} c^\nu_{(\lambda,\tau)} s_\lambda. \end{equation}
By~\eqref{eq:lrProd} and~\eqref{eq:lrSkew}, we have 
\begin{equation} \begin{split} 
 \langle s_\tau s_\lambda, s_\nu \rangle =c^{\lambda}_{(\nu/\tau)} &=|\SSYTL(\nu/\tau, \lambda)|
 \\
 &\quad = |\mSSYTL\bigl( (\lambda, \tau), \nu \bigr)|  = c^\nu_{(\lambda,\tau)} = \langle s_\lambda, s_{\nu/\tau} \rangle \end{split}\label{eq:lrRel} \end{equation}
where the middle equality is proved in Proposition~\ref{prop:shapecontent} in the appendix.
%\[ \begin{split} \langle s_\lambda, s_{\nu/\tau} \rangle = c^\nu_{(\lambda,\tau)} = 
%|\mSSYTL\bigl(& (\lambda, \tau), \nu \bigr)| \\
%&= |\SSYTL(\nu/\tau, \lambda)| = c^{\lambda}_{(\nu/\tau)}
%= \langle s_\tau s_\lambda, s_\nu \rangle.\end{split} \]
%It therefore follows from~\eqref{eq:lrProd} and~\eqref{eq:lrSkew} that
%$\langle s_\tau s_\lambda, s_\nu \rangle = \langle s_\lambda, s_{\nu/\tau} \rangle $.
This gives a  combinatorial proof of the
fundamental adjointness relation for Schur functions. 
%(See
%\cite[7.15.4]{StanleyII} for a more conventional proof.)
By~\eqref{eq:lrGen} and this relation 
we have $\langle  s_{\lambda/\mu}, s_{\nu/\tau} \rangle = c^{\lambda}_{(\nu/\tau,\mu)}$.
If~$\mathbf{t}$ is a latticed skew $2$-multitableaux
of shape $(\nu/\tau, \mu)$ then, as seen in~\eqref{eq:finalFixed},
$\mathbf{t} = (t,u(\mu))$ for some $\nu/\tau$-tableau $t$. Thus
\begin{equation} \langle s_{\nu/\tau}, s_{\lambda/\mu} \rangle = c^{\lambda}_{(\nu/\tau,\mu)} =
\bigl| \bigl\{ t \in \SSYT(\nu/\tau, \lambda - \mu) : \text{$\bigl( t,u(\mu) \bigr)$ is latticed} \bigr\} \bigr|. \label{eq:lrSkewSkew} \end{equation}
This is equivalent to the skew-skew Littlewood--Richardson rule proved in \cite[\S 4]{RemmelShimozono}.
The non-obvious equalities
 $|\SSYTL(\nu/\tau,\lambda)| = |\SSYTL(\nu/\lambda,\tau)|$ and
 $\bigl|\mSSYTL\bigl( (\lambda,\tau), \nu\bigr)\bigr| = 
 \bigl|\mSSYTL\bigl( (\tau,\lambda), \nu\bigr)\bigr|$
are also corollaries of~\eqref{eq:lrRel}.

As a final exercise, we show that our definition of generalized Littlewood--Richardson coefficients
is consistent with the algebraic generalisation of~\eqref{eq:lrGen}
%the equality $\langle s_\tau s_\lambda, s_\nu \rangle = c^{\nu}_{(\lambda,\tau)}$ in~\eqref{eq:lrRel}
to arbitrary products of Schur functions.

\begin{lemma}\label{lemma:assoc}
Let $m \in \N$.
If $\sigmanub/\taub$ is a skew $m$-multipartition of $n \in \N_0$ and~$\lambda$ is a partition of $n$ then
\[ s_{\sigmanu(1)/\tau(1)} \ldots s_{\sigmanu(m)/\tau(m)} = 
\sum_{\lambda} c^{\lambda}_{\sigmanub/\taub} \]
where the sum is over all partitions $\lambda$ of $n$.
%{(\sigma(1)/\tau(1),\ldots,\sigma(m)/\tau(m))} s_{\lambda}. \]
%where the sum is over all partitions $\lambda$ of $n$.
\end{lemma}
\begin{proof}
%For ease of notation we give the proof in the case $s=3$. It is easily generalized.
%By~\eqref{eq:LR} we have
%\[
%\langle s_{\sigma(1)} \ldots s_\sigma(s), s_\lambda\rangle = \sum_{\gamma(1), \gamma(2), \sigma(3)} c^{\gamma(1)}_{\gamma(2)\sigma(1)} \ldots 
%c^{\gamma(s-2)}_{\gamma(s-1)\sigma(s-2)} c^{\gamma(s-1)}_{\gamma(s)\sigma(s-1)} \]
By induction, the fundamental adjointness relation and~\eqref{eq:lrSkewSkew} we have
\begin{align*}
\langle s_{\sigmanu(1)/\tau(1)} s&_{\sigmanu(2)/\tau(2)} \ldots s_{\sigmanu(m)/\tau(m)}, s_\lambda \rangle \\
&\ \ =
\langle \sum_{\gamma} s_{\sigmanu(1)/\tau(1)} c_{((\sigmanu(2)/\tau(2), \ldots, \sigmanu(m)/\tau(m))}^{\gamma}
s_\gamma, s_\lambda \rangle \\
&\ \ = \sum_{\gamma} \langle s_{\sigmanu(1)/\tau(1)}, s_{\lambda/\gamma} \rangle 
c_{((\sigmanu(2)/\tau(2), \ldots, \sigmanu(m)/\tau(m))}^{\gamma} \\
&\ \ = \sum_{\gamma}  c_{(\sigmanu(1)/\tau(1),\gamma)}^{\lambda} \,
c_{((\sigmanu(2)/\tau(2), \ldots, \sigmanu(m)/\tau(m))}^{\gamma}
\end{align*}
where the sums are over all partitions $\gamma$ of $n-(|\sigmanu(1)|-|\tau(1)|)$.
The right-hand side
counts the number of pairs of  semistandard skew multitableaux
$\bigl( \bigl( t, u(\gamma) \bigr), \mathbf{t}\bigr)$ such that
$t \in \SSYTL(\sigmanu(1)/\tau(1), \lambda-\gamma)$ and $\mathbf{t} \in
\mSSYTL\bigl( \bigl(\sigmanu(2)/\tau(2),\ldots,\sigmanu(m)/\tau(m)\bigr),\gamma\bigr)$.
Such pairs are in bijection with $\mSSYTL\bigl(
\bigl( \sigmanu(1)/\tau(1),\ldots,\sigmanu(m)/\tau(m)\bigr),\lambda\bigr)$
by the map sending $\bigl( \bigl( t, u(\gamma) \bigr), \mathbf{t}\bigr)$ to the
concatenation $(t : \mathbf{t})$.
The lemma follows.
\end{proof}

\subsection{Plethystic rules}\label{subsec:plethystic}

By Theorem~\ref{thm:main} and the fundamental adjointness relation, we have
$\langle  s_{\lambda} \circ p_r,  s_{\nu/\tau} \rangle = \sgn_r(\nub/\taub) 
c^\lambda_{\nub/\taub}$.
% if $\nu/\tau$ is $r$-decomposable (in the sense
%defined in \S\ref{sec:background}), and 
%$\langle  s_{\lambda} \circ p_r,  s_{\nu/\tau} \rangle = 0$ otherwise.
Hence, by Lemma~\ref{lemma:assoc},
\begin{equation}\label{eq:padjoint} \langle s_{\lambda} \circ p_r, s_{\nu/\tau} \rangle = 
\begin{cases} \langle
s_\lambda, s_{\nu(0)/\tau(0)} \ldots s_{\nu(r-1)/\tau(r-1)} \rangle 
& \text{if $\nu/\tau$ is $r$-decomposable}
\\
0 & \text{otherwise.}\end{cases} \end{equation}
This adjointness relation was first
proved in \cite{KSW}: for a more recent proof see \cite[after (39)]{DesarmenienLeclercThibon}.
It is perhaps 
a little surprising that~\eqref{eq:padjoint}
implies that the absolute value of the 
coefficient
of $s_{(\nub/\taub,\tau)^\star}$ in $s_\tau(s_\lambda \circ p_r)$, namely
%\[ 
$c_{\nub/\taub}^\lambda = |\mSSYTL(\nub/\taub, \lambda)|$, % \]
is the same for all $r!$ permutations of the $r$-quotient $\nub/\taub$.

Note that we obtain only a numerical equality: even cyclic permutations 
of skew $r$-multitableaux, do not,
in general preserve the lattice property. For example,
changing the abaci in Figure~2 in \S\ref{sec:example}
so that $7$ beads are used to represent $(3,2)$ and $(6,5,5,5,2)$
induces a rightward cyclic shift of the skew tableaux forming the skew $3$-multitableaux 
$\mathbf{t}_1, \mathbf{t}_2, \mathbf{t}_3, \mathbf{t}_4$. 
After one or two such shifts,
the unique latticed skew $3$-multitableaux are
the shifts of~$\mathbf{t}_3$ and~$\mathbf{t}_2$, respectively;
$\mathbf{t}_4$ remains unlatticed after any number of shifts. 
The identification of $\mathbf{t}_1$ 
as the unique skew $3$-multitableau contributing to the 
coefficient of $s_{(6,5,5,5,2)}$ in $s_{(3,2)}(s_{(3,3)} \circ p_3)$
is therefore canonical, but not entirely natural.

The author is aware of two combinatorial
rules in the literature for special cases of the product $s_\tau(s_\lambda \circ p_r)$
that avoid this undesirable feature of the SXP rule.
To state the first, which is due to Carr{\'e} and Leclerc,  we 
need a definition from \cite{CarreLeclerc}. Let $T$ be an $r$-ribbon
tableau of shape $\nu/\tau$ and weight~$\lambda$. % as in~\eqref{eq:ribbons}. 
Represent $T$, as in FIgure~3, by a tableau of shape $\nu/\tau$ in which
the boxes of the $\alpha_j$ disjoint $r$-border strips forming the horizontal $r$-ribbon in $T$
labelled~$j$
all contain~$j$.
The \emph{column word}
of~$T$ is the word of length~$n$
obtained by reading the columns of this tableau from bottom to top, starting at 
the leftmost column, and recording the label of 
each $r$-border strip when it is first seen, in its leftmost column.

\begin{theorem}[\protect{\cite[Corollary 4.3]{CarreLeclerc}}]\label{thm:CL}
Let $r \in \N$ and let $n \in \N_0$.
Let $\nu / \tau$ be a skew partition of $rn$ and let~$\lambda$ be a partition of $n$.
Up to the sign $\sgn_2(\nu/\tau)$, the multiplicity $\langle s_\tau(s_\lambda \circ p_2), s_\nu 
\rangle$ is equal to the number of $2$-ribbon tableaux~$T$ of shape $\nu/\tau$ and weight $\lambda$
whose column word is latticed.
\end{theorem}

% These $2$-ribbon tableaux are called \emph{domino tableaux}
%in \cite{CarreLeclerc}.
For example, there are two
$2$-ribbon tableaux of shape $(5,5,2,2)/(3,1)$ and content $(3,1,1)$
having a latticed column word 
(see Figure~4 overleaf), and so
 $\langle s_{(3,1)}(s_{(3,1,1)} \circ p_2), s_{(5,5,2,2)} \rangle = 2$.
The corresponding skew $2$-multitableaux of shape $\bigl( (3,1)/(2), (2,1) \bigr)$ 
are
\[ \Bigl(\, \young(::1,1)\,,\ \young(12,3)\, \Bigr),\quad  
   \Bigl(\, \young(::1,3)\,,\ \young(11,2)\, \Bigr), \]
respectively. Only the second is latticed in the multitableau sense.

\begin{figure}[t]
\begin{center}
\includegraphics{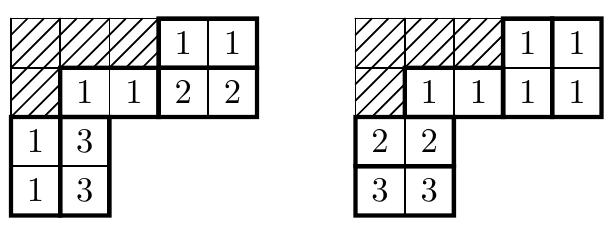}
\end{center}
\caption{\small The two $2$-ribbon tableaux of shape $(5,5,2,2)/(3,1)$ and content $(3,1,1)$
whose column words, namely $13121$ and $32111$, are latticed.}
\end{figure}

In Theorem 6.3 of \cite{EvseevPagetWildon}, Evseev, Paget and the  author
applied character theoretic arguments to the case $\lambda = (a,1^b)$,
considering arbitrary $r \in \N$. To restate this result in our setting, we 
introduce the following definition.

\begin{definition}
The \emph{row-number tableau}  of an $r$-ribbon tableau $T$ is 
the row-standard tableau $\RNT(T)$ defined
by putting an entry $i$ in row $a$ of $\RNT(T)$ for each
$r$-border strip of row number $a$ in the $r$-ribbon strip of $T$ labelled~$i$.
\end{definition}

If $T$ has weight $\lambda$ then the content of $\RNT(T)$ is $\lambda$. 
The shape of $\RNT(T)$ is in general a composition, possibly with some zero parts.
The row-number tableaux of the four $3$-ribbon tableaux in 
$3$-$\RT\bigl( (6,5,5,5,2)/(3,2),(3,3) \bigr)$, shown in the top line of Figure~3 in \S\ref{sec:example},
are
\[ \begin{matrix}\young(1,222,11)\\[23pt] \end{matrix}\,,  \qquad \begin{matrix}\young(11)\hfill\\[13pt]\young(122,2)\end{matrix}\,,\qquad
\young(1,122,1,2)\,,    \qquad \begin{matrix}\young(11)\hfill\\[13pt]\young(1222)\\[11pt]\end{matrix}\,. \]

The definition of latticed extends to row-number tableaux in the obvious way.
The second row-number tableau above, with word $212211$, is the only one that is latticed.
%In Figure~4, the row-number tableaux 

\begin{corollary}[see \protect{\cite[Theorem 6.3]{EvseevPagetWildon}}]\label{cor:EPW}
Let $r \in \N$, let $a \in \N$ and let $b \in \N_0$. Let $\nu/\tau$ be a skew partition of $r(a+b)$.
Then
$\langle s_\tau (s_{(a,1^b)} \circ p_r), s_\nu \rangle$
is equal, up to the sign $\sgn_r(\nu/\tau)$, 
to the number of $r$-ribbon tableaux of shape $\nu/\tau$ and weight $(a,1^b)$ whose
row-number tableau is latticed. The column word of such an $r$-ribbon tableau is $(b+1)b \ldots 
21\ldots 1$,   where the number of~$1$s is~$a$.
\end{corollary}

\begin{proof}
By Theorem 6.3 in \cite{EvseevPagetWildon}, up to
the sign $\sgn_r(\nu/\tau)$, the multiplicity $\langle s_\tau (s_{(a,1^b)} \circ p_r), s_\nu \rangle$ 
is the number of
$(a,1^b)$-like border-strip $r$-diagrams of shape $\nu/\tau$, 
as defined in \cite[Definition 6.2]{EvseevPagetWildon}.  
(The required translation
from character theory to symmetric functions is outlined in \cite[\S 7]{EvseevPagetWildon}.)
To relate these objects to $r$-ribbon tableaux, %it will be convenient to 
we define a skew partition $\rho/\tau$ to be a
\emph{vertical $r$-ribbon strip} if $\rho'/\tau'$ is a horizontal $r$-ribbon strip.

Let $T$ be an $r$-ribbon tableau of shape $\nu/\tau$ and weight $(a,1^b)$. There 
is a unique partition $\rho$ such that $\rho/\tau$ is the horizontal $a$-ribbon
strip in $T$ and $\nu/\rho$ is a vertical $b$-ribbon strip, formed from the border strips
labelled $2$, \ldots, $b+1$. Suppose $\RNT(T)$ is latticed. Then 
the row numbers of these border strips are increasing. Moreover,
the rightmost border strip in either of the ribbons $\rho/\tau$ and $\nu/\rho$ 
lies in the Young diagram of $\rho/\tau$,
and the skew partition formed from this border strip and $\nu/\rho$ is a vertical $(b+1)$-ribbon strip.
Therefore $T$ corresponds to an $(a,1^b)$-like border-strip $r$-diagram of shape~$\nu/\tau$,
and the column word of $T$ is as claimed.
 Conversely, each such $r$-ribbon tableau arises in this way.
\end{proof}

The second claim in Corollary~\ref{cor:EPW} implies that
if $T$ is an $r$-ribbon tableau of weight $(a,1^b)$ whose row-number tableau $\RNT(T)$ is latticed,
then the word of $\RNT(T)$ agrees with the column word of $T$.
Hence the combinatorial rules for  $\langle s_\tau (s_{(a,1^b)} \circ p_2), s_\nu \rangle$ 
obtained by taking $\lambda = (a,1^b)$ 
in Corollary 4.3 of \cite{CarreLeclerc} or $r=2$ in Corollary~\ref{cor:EPW}
count the same sets of $r$-ribbon tableaux.
For example, in Figure~4 we have $a=3$ and $b=2$; 
the first $2$-ribbon tableau 
has a horizontal $2$-ribbon strip of shape $(5,3,1,1)/ (3,1)$, a vertical
$2$-ribbon strip of shape  $(5,5,2,2) / (5,3,1,1)$, and the augmented vertical
$2$-ribbon strip has shape $(5,5,2,2) / (3,3,1,1)$. 

For general weights we have the following result. 

\begin{proposition}\label{prop:RNTCW}
Let $r \in \N$. Let $T$ be a 
$r$-ribbon tableau. If the column word of $T$ is latticed then the row-number
tableau of $T$ is latticed.
\end{proposition}

The proof is given in the appendix. %Rows $1$ and $2$ of the table overleaf show that
The converse of Proposition~\ref{prop:RNTCW} is false.
For example $\langle s_{(2,2)} \circ p_3, s_{(3,3,3,3)} \rangle = 1$. 
The two $3$-ribbon tableau in $\rRT\bigl( (3,3,3,3), (2,2) \bigr)$ are shown below.
Both have a latticed
row-number tableau, with word $2211$. The column words are $2112$ and $2121$ respectively;
only the second is latticed.

\begin{center}
\includegraphics{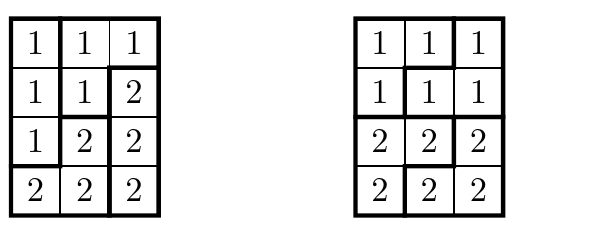}
\end{center}

In both Theorem~\ref{thm:CL} and Corollary~\ref{cor:EPW}
there is a lattice condition that refers directly to certain sets of
$r$-ribbon tableaux, without making use of $r$-quotients. In the following
problem, which the author believes is open under its intended interpretation (except when $r = 2$
or $\lambda = (a,1^b)$ for some $a \in \N$ and $b \in \N_0$) we say that such conditions are \emph{global}.

\begin{problem}\label{prob:gen}
Find a combinatorial rule, simultaneously generalizing
Theorem~\ref{thm:CL} and Corollary~\ref{cor:EPW},
that expresses 
$\langle s_\tau(s_\lambda \circ p_r), s_\nu \rangle$ as the product of $\sgn_r(\nu/\tau)$ and  
the size of a set of $r$-ribbon tableaux of shape $\nu/\tau$ satisfying
a global lattice condition.
\end{problem}
	
The obvious generalizations of Theorem~\ref{thm:CL} and Corollary~\ref{cor:EPW} fail even
to give correct upper and lower bounds on the multiplicity in Problem~\ref{prob:gen}.
Counterexamples are shown in the table below.
The second column gives the number of $r$-ribbon tableaux of the
relevant shape and weight and the final
two columns count those $r$-ribbon tableaux 
whose column word is latticed (CWL), and whose row-number tableaux is latticed (RNTL), respectively.

\medskip
\begin{center}
\begin{tabular}{cccc}\toprule 
plethysm $\langle s_\tau(s_\lambda \circ p_r), s_\nu \rangle$ 
& $|r$-$\RT(\nu/\tau,\lambda)|$ & CWL & RNTL  \\ \midrule
%$\langle s_{(2,2)} \circ p_3, s_{(3,3,3,3)} \rangle = 1$ & 2 & 1 & 2 \\
$\langle s_{(3,3)} \circ p_3, s_{(6,6,6)} \rangle = 1$ & 6 & 0 & 2 \\
$\langle s_{(2,2,2)} \circ p_4, s_{(7,4,4,4,4,1)} \rangle = -1$ & 9 & 0 & 0 \\
$\langle s_{(1)}(s_{(3,3)} \circ p_3), s_{(6,6,6,1)} \rangle = 1$ & 6 & 0 & 0  \\
$\langle s_{(1)}(s_{(2,2)} \circ p_4), s_{(5,4,4,4)} \rangle = 1$ & 2 & 2 & 2 
\\ \bottomrule \end{tabular}
\end{center}

\medskip

Despite this, there are some signs that row-number tableaux are a useful object in more
general settings than Corollary~\ref{cor:EPW}. In particular, 
the following conjecture holds when $r \le 4$ and $n \le 10$ and when $r \le 6$ and $n \le 6$.
(Haskell \cite{Haskell98} source code to verify this claim % and the table above
is available from the author.)
When $r=2$ it holds by Theorem~\ref{thm:CL} and Proposition~\ref{prop:RNTCW},
replacing $(a,b)$ with a general partition $\lambda$; by row $2$ of the table
above, this more general conjecture is false when $r=4$. By Corollary~\ref{cor:EPW},
the conjecture holds, with equality, when $b=1$.
%When $r=3$, the more
%general conjecture holds whenever $|\lambda| \le 8$.

\begin{conjecture}\label{conj:RNT}
Let $r \in \N$, let $n \in \N_0$, let $\nu$ be a partition of $rn$ and let $(a,b)$ be a partition of $n$.
The number of $r$-ribbon tableaux~$\,T$ of shape $\nu$ and weight~$(a,b)$ such
that the row-number tableau $\RNT(T)$ is latticed is an upper bound for the absolute
value of $\langle s_{(a,b)} \circ p_r, s_\nu \rangle$.
\end{conjecture}

\section*{Appendix: the shape-content involution and proof of Proposition~\ref{prop:RNTCW}}
\setcounter{section}{7}

In the proof of~\eqref{eq:lrSkewSkew} we used the following proposition.
\setcounter{theorem}{0}

\begin{proposition}\label{prop:shapecontent}
If $\lambda$, $\mu$ and $\nu$ are partitions then 
\[ \bigl| \bigl\{ t \in \SSYT(\nu, \lambda - \mu) : \text{$\bigl( t,u(\mu) \bigr)$ is latticed} \bigr\} \bigr| = \bigl|\SSYTL(\lambda/\mu, \nu)\bigr|. \]
%is equal to the number of latticed semistandard tableaux of shape $\lambda / \mu$.
\end{proposition}

This proposition follows immediately from Lemma~\ref{lemma:shapecontent}(iii) below, 
by setting $\alpha = \nu$ and $\beta=
\varnothing$. The `shape-content involution' given in this 
lemma is surely well known to experts, but the author has not found it
in the literature in this generality. The lemma may also be used to show that the final
corollary in Stembridge's involutive proof of the Littlewood--Richardson rule \cite{StembridgeLR}
is equivalent to~\eqref{eq:lrProd}; this is left as a `not-too-difficult exercise' in \cite{StembridgeLR}.

%\subsubsection*{The shape-content involution}
Let $\lambda/\mu$ and $\alpha/\beta$ be skew partitions of the same size. Let
$\RSYT(\lambda/\mu,\alpha/\beta)$ be the set of all row-standard $\lambda/\mu$ tableaux $t$
such that $\beta + \cont(t) = \alpha$. Let
\begin{align*} \RSYTL(\lambda/\mu, \alpha/\beta) &= \bigl\{ t \in \RSYT(\lambda/\mu, \alpha/\beta) : 
\text{$\bigl( t, u(\beta) \bigr)$ is latticed} \bigr\}, \\
\SSYTL(\lambda/\mu, \alpha/\beta) &= \RSYTL(\lambda/\mu,\alpha/\beta) \cap \SSYT(\lambda/\mu,\alpha/\beta). 
%\{ t \in \RSYTL(\lambda/\mu, \alpha/\beta) : 
%\text{$\bigl( t, u(\beta) \bigr)$ is latticed and semistandard} \}
\end{align*}
%(This is consistent with the definition of latticed semistandard multitableaux in \S\ref{sec:LS},
%since if  $\mu=\beta = \varnothing$ then $t \in \SSYTL(\lambda/\mu, \alpha)$ if and only
%if the $1$-multitableau $(t)$ is in $\mSSYTL\bigl( (\lambda), \alpha\bigr)$.)
Given $t \in \RSYT(\lambda/\mu,\alpha/\beta)$, let $\CC(t)$ be the
row-standard tableau of shape $\alpha/\beta$ defined
by putting a $k$ in row $a$ of $\CC(t)$ for every $a$ in row $k$ of~$t$. 

\bigskip
%\samepage{
\begin{lemma}[Shape/content involution]{ \ }\label{lemma:shapecontent}
\begin{thmlist}
\item $\CC : \RSYT(\lambda/\mu,\alpha/\beta) \rightarrow \RSYT(\alpha/\beta,\lambda/\mu)$ is an involution.
\item $\CC$ restricts to a involution $\SSYT(\lambda/\mu,\alpha/\beta) \rightarrow \RSYTL(\alpha/\beta,
\lambda/\mu)$.
\item $\CC$ restricts to a involution $\SSYTL(\lambda/\mu,\alpha/\beta) \rightarrow \SSYTL(\alpha/\beta,
\lambda/\mu)$.
%\[ \SSYT(\lambda/\mu,\alpha/\beta) \rightarrow 
%\bigl\{t \in \RSYT(\alpha/\beta,\lambda/\mu) : \text{$\bigl(t, u(\mu) \bigr)$ is latticed} \bigr\}. \]
%\item $C$ restricts further to a bijection 
%\[ \bigl\{t \in \RSYT(\lambda/\mu,\alpha/\beta) :\text{$\bigl(t, u(\beta) \bigr)$ is latticed} \bigr\}
%\rightarrow
%\bigl\{t \in \RSYT(\alpha/\beta,\lambda/\mu) : \text{$\bigl(t, u(\mu) \bigr)$ is latticed} \bigr\}. \]
\end{thmlist}
\end{lemma} %}

\begin{proof}
(i) is obvious. For (ii) observe that if $t \in \RSYTL(\lambda/\mu, \alpha/\beta)$
then $\bigl( t, u(\beta) \bigr)$ is not latticed
if and only if there exists $k \in \N$ and
an entry $k+1$ in row $a$ of $t$ and position $i$ of $\w(t)$ such that
\[ \big| \{j : \w(t)_j = k+1, j \ge i \} \bigr| + \beta_{k+1} 
= \bigl| \{j : \w(t)_j = k : j > i \} \bigr| + \beta_k + 1. \]
Let $b$ be the common value.
The first $b-1-\beta_k$ entries in row $k$ of $\CC(t)$ are at most $a-1$,
and the next entry is the number of a row $a'$ with $a' \ge a$.
The entry below is the $(b-\beta_{k+1})$-th entry in row $k+1$ of $\CC(t+1)$, namely~$a$.
Therefore $\CC(t)$ is not semistandard. The converse may be proved by reversing this argument.
It follows from (ii) that~$\CC$ restricts to  involutions
$\SSYT(\lambda/\mu,\alpha/\beta) \rightarrow \RSYTL(\alpha/\beta,
\lambda/\mu)$ and 
$\RSYTL(\lambda/\mu,\alpha/\beta) \rightarrow \SSYT(\alpha/\beta, \lambda/\mu)$;
taking the common domain and codomain of these involutions we get~(iii).
\end{proof}

We end with the proof of Proposition~\ref{prop:RNTCW}. One final definition will be useful.
Let $D$ be a subset of the boxes of a Young diagram of a partition $\nu$.
 If column $b$ is the least
numbered column of $\nu$ meeting $D$, 
then we say that $D$ has \emph{column number} $b$,
and write $C(D) = b$. (Thus if $D$ is a border strip in $\nu$ then
$D$ has column number $b$ if and only if the conjugate border strip $D'$ in $\nu'$ has row number~$b$.)
For an example see Figure~5 overleaf.

\begin{proof}[Proof of Proposition~\ref{prop:RNTCW}]
Suppose that the labels of the $r$-ribbons in $T$ are $\{1,\ldots, \ell\}$. Fix \hbox{$k < \ell$}.
Let $D_1, \ldots, D_q$ be the subsets of the Young diagram of $\nu/\tau$ that form
the $r$-border strips lying in the $r$-ribbon strips of $T$
labelled~$k$ and $k+1$, written in the order corresponding to the column word of $T$. Thus
\[ C(D_1) \le \ldots \le C(D_q) \]
and if $C(D_j) = C(D_{j+1})$ then $R(D_j) > R(D_{j+1})$. Let $N(D_j) \in \{k,k+1\}$ be
the label of $D_j$. Let
\[  w = N(D_1) N(D_2) \ldots N(D_q) \]
be the subword of the column word of $T$ formed from the entries $k$ and $k+1$.
By hypothesis, $w$ has no $k$-unpaired $k+1$.

Let $v$ be the subword 
of the word of the row-number tableau $\RNT(T)$ formed from the entries $k$ and $k+1$.
We may obtain $v$ by reading the rows
of~$T$ from left to right, starting at the highest numbered row, and writing down
the label $N(D_j)$ of $D_j$ on the final occasion when we see a box of $D_j$.
By~\eqref{eq:rownumbers}, if $N(D_j) = N(D_{j+1})$ then $R(D_j) \ge R(D_{j+1})$. 
Moreover, if $N(D_j) = k+1$ and $N(D_{j+1}) = k$ then $R(D_j) > R(D_{j+1})$. 
Therefore
$N(D_j)$ is written after $N(D_{j+1})$ when writing $v$ if and only
if $N(D_j) = k$, $N(D_{j+1}) = k+1$ and $R(D_j) < R(D_{j+1})$. We say that such $j$
are \emph{inversions}. If there are no inversions,
then $v$ and $w$ are equal. Otherwise,
let~$j$ be minimal such that $j$ is an inversion,
and let~$s$ be maximal such that $R(D_j) < R(D_{j+s})$; note that $N(D_{j+s}) = k+1$,
by~\eqref{eq:rownumbers}. The word $v$ is  obtained from $w$
sorting its entries in positions $j, j+1, \ldots, j+s$ into decreasing order,
and then continuing inductively with the later positions. It is clear 
that this procedure does not create a new $k$-unpaired $k+1$. Hence $v$ has no $k$-unpaired $k+1$.
\end{proof}

\begin{figure}
\begin{center}
\includegraphics{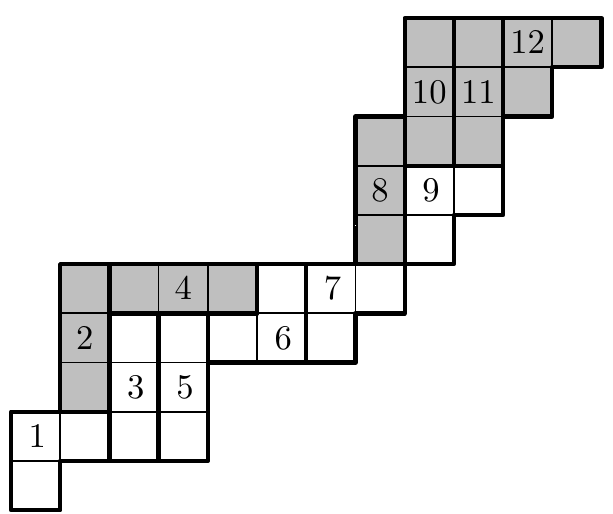}
\end{center}
\caption{Border strips $D_1,\ldots, D_{12}$ labelled $k$ (grey) 
or $k+1$ (white) forming the $3$-ribbons
in a $3$-ribbon tableau $T$ are shown. Numbers
are as in the proof of Proposition~\ref{prop:RNTCW}.
For example, $R(D_9) = 4$ and $C(D_9)= 9$.
The subword of the column word with entries $k$ and $k+1$ 
is $k^+ k k^+ k k^+ k^+ k^+ k k^+ k k k$, where $k^+$ denotes $k+1$. 
The inversions are $2$, $4$ and $8$.
%$j$ such that $N(D_j) = k$ and $N(D_{j+1}) = k+1$ and $R(D_j) < R(D_{j+1})$
% are $2, 4, 8$.
%as defined in the proof
%of Proposition~\ref{prop:RNTCW}. 
The subword of the row word of the row-number tableau of~$T$ with entries $k$ and $k^+$
is obtained by sorting the entries in positions $2,3,4,5$ and $8,9$ into decreasing order,
giving $k^+ k^+ k^+ k k k^+ k^+ k^+ kkkk$.}
\end{figure}

\def\cprime{$'$} \def\Dbar{\leavevmode\lower.6ex\hbox to 0pt{\hskip-.23ex
  \accent"16\hss}D} \def\cprime{$'$}
\providecommand{\bysame}{\leavevmode\hbox to3em{\hrulefill}\thinspace}
\renewcommand{\MR}[1]{\relax}


\begin{thebibliography}{10}

\bibitem{CarreLeclerc}
C. Carr{\'e} \and B. Leclerc,
\emph{Splitting the square of a Schur function into its symmetric and antisymmetric parts},
J. Algebraic Combin. \textbf{4} (1995), 201--231. 

\bibitem{ChenGarsiaRemmel}
Y.~M. Chen, A.~M. Garsia, and J.~Remmel, \emph{Algorithms for plethysm},
  Combinatorics and algebra ({B}oulder, {C}olo., 1983), Contemp. Math.,
  vol.~34, Amer. Math. Soc., Providence, RI, 1984, pp.~109--153. 
  
\bibitem{DesarmenienLeclercThibon}
J.~D{\'e}sarm{\'e}nien, B.~Leclerc, and J.-Y. Thibon, \emph{Hall-{L}ittlewood
  functions and {K}ostka-{F}oulkes polynomials in representation theory},
  S\'em. Lothar. Combin. \textbf{32} (1994), Art.\ B32c, approx. 38 pp.
 
 \bibitem{EvseevPagetWildon}
A.~Evseev, R.~Paget and M.~Wildon, \emph{Character deflations and a generalization
  of the Murnaghan--Nakayama rule}, J. Group Theory \textbf{17} (2014), 1034--1070.


\bibitem{JK}
G.~James and A.~Kerber, \emph{The representation theory of the symmetric
  group}, Encyclopedia of Mathematics and its Applications, vol.~16,
  Addison-Wesley Publishing Co., Reading, Mass., 1981.
  
\bibitem{James}
G.~D. James, \emph{The representation theory of the symmetric groups}, Lecture
  Notes in Mathematics, vol. 682, Springer, Berlin, 1978. 
 

\bibitem{KSW}
{A. Kerber, F. S\"anger and B. Wagner},
\emph{Quotienten und Kerne von Young-Diagrammen, Brettspiele und Plethysmen gew{\"o}hnlicher irreduzibler Darstellungen symmetrischer Gruppen},
 Mitt. Math. Sem. Giessen \textbf{149} (1981), 131--175. 
 
  
\bibitem{LascouxSchutzenberger}
A. Lascoux and M.-P. Sch{\"u}tzenberger, \emph{Le mono\"\i de plaxique},
  Noncommutative structures in algebra and geometric combinatorics ({N}aples,
  1978), Quad. ``Ricerca Sci.'', vol. 109, CNR, Rome, 1981, pp.~129--156.


\bibitem{LittlewoodModular}
D.~E. Littlewood, \emph{Modular representations of symmetric groups}, Proc.
  Roy. Soc. London. Ser. A. \textbf{209} (1951), 333--353.
  
\bibitem{LittlewoodRichardson}
D.~E. Littlewood and A.~R. Richardson, \emph{Group characters and algebra},
  Phil. Trans. Roy. Soc. Lond. Ser. A \textbf{233} (1934), 99--141.

\bibitem{LoehrAbacus}
N.~A.~Loehr, \emph{Abacus proofs of {S}chur function identities}, SIAM J.
  Discrete Math., \textbf{24} (2010), 1356--1370. 

\bibitem{LoehrRemmel}
N.~A. Loehr and J.~B. Remmel, \emph{A computational and
  combinatorial expos\'e of plethystic calculus}, J. Algebraic Combin.
  \textbf{33} (2011), no.~2, 163--198. 
  
\bibitem{Lothaire}
M.~Lothaire, \emph{Algebraic combinatorics on words}, Encyclopedia of
  Mathematics and its Applications, vol.~90, Cambridge University Press,
  Cambridge, 2002, A collective work by Jean Berstel, Dominique Perrin, Patrice
  Seebold, Julien Cassaigne, Aldo De Luca, Steffano Varricchio, Alain Lascoux,
  Bernard Leclerc, Jean-Yves Thibon, Veronique Bruyere, Christiane Frougny,
  Filippo Mignosi, Antonio Restivo, Christophe Reutenauer, Dominique Foata,
  Guo-Niu Han, Jacques Desarmenien, Volker Diekert, Tero Harju, Juhani
  Karhumaki and Wojciech Plandowski, with a preface by Berstel and Perrin.


\bibitem{MacDonald}
I.~G. Macdonald, \emph{Symmetric functions and {H}all polynomials}, second ed.,
  Oxford Mathematical Monographs, The Clarendon Press Oxford University Press,
  New York, 1995, With contributions by A. Zelevinsky, Oxford Science
  Publications.
  
\bibitem{Muir}
T. Muir, \emph{On the quotient of a simple alternant by the difference-product of the variables},
Proc. Roy. Soc. Edin. \textbf{14} (1888), 433--445.



\bibitem{Haskell98}
S. {Peyton Jones} et~al., \emph{The {Haskell} 98 language and libraries: The
  revised report}, Journal of Functional Programming \textbf{13} (2003), 
  0--255, \url{http://www.haskell.org/definition/}.

\bibitem{RemmelShimozono}
J.~B. Remmel and M.~Shimozono, \emph{A simple proof of the
  {L}ittlewood-{R}ichardson rule and applications}, Discrete Math. \textbf{193}
  (1998), no.~1-3, 257--266, Selected papers in honor of Adriano Garsia
  (Taormina, 1994). 

\bibitem{StanleyII}
R.~P. Stanley, \emph{Enumerative combinatorics. {V}ol. 2}, Cambridge
  Studies in Advanced Mathematics, vol.~62, Cambridge University Press,
  Cambridge, 1999, with a foreword by Gian-Carlo Rota and appendix 1 by Sergey
  Fomin. 

\bibitem{StembridgeLR}
J.~R. Stembridge, \emph{A concise proof of the {L}ittlewood-{R}ichardson
  rule}, Electron. J. Combin. \textbf{9} (2002), Note 5, 4 pp.
  (electronic).

\bibitem{WildonPlethysticMN}
M.~Wildon, \emph{A combinatorial proof of a plethystic Murnaghan--Nakayama
  rule}, SIAM J. Discrete Math., \textbf{30} (2016), 1526--1533.


\end{thebibliography}
\end{document}